\documentclass[reqno,a4paper]{amsart}
\usepackage{amssymb}
\usepackage{amsmath}
\begin{document} 
\newtheorem{prop}{Proposition}[section]
\newtheorem{Def}{Definition}[section]
\newtheorem{theorem}{Theorem}[section]
\newtheorem{lemma}{Lemma}[section] \newtheorem{Cor}{Corollary}[section]

\title[Yang-Mills and Yang-Mills-Higgs]{\bf Local well-posedness for the (n+1)-dimensional Yang-Mills and Yang-Mills-Higgs system in temporal gauge}
\author[Hartmut Pecher]{
{\bf Hartmut Pecher}\\
Fakult\"at f\"ur  Mathematik und Naturwissenschaften\\
Bergische Universit\"at Wuppertal\\
Gau{\ss}str.  20\\
42119 Wuppertal\\
Germany\\
e-mail {\tt pecher@math.uni-wuppertal.de}}
\date{}

\begin{abstract}
The Yang-Mills and Yang-Mills-Higgs equations in temporal gauge are locally well-posed for small and rough initial data, which can be shown using the null structure of the critical bilinear terms. This carries over a similar result by Tao for the Yang-Mills equations in the (3+1)-dimensional case to the more general Yang-Mills-Higgs system and to general dimensions.
\end{abstract}
\maketitle
\renewcommand{\thefootnote}{\fnsymbol{footnote}}
\footnotetext{\hspace{-1.5em}{\it 2010 Mathematics Subject Classification:} 
35Q40, 35L70 }
\normalsize 
\setcounter{section}{0}
\section{Introduction and main results}
\noindent 
Let $\mathcal{G}$ be the Lie goup $SO(n,\mathbb{R})$ (the group of orthogonal matrices of determinant 1) or $SU(n,\mathbb{C})$ (the group of unitary matrices of determinant 1) and $g$ its Lie algebra $so(n,\mathbb{R})$ (the algebra of trace-free skew symmetric matrices) or $su(n,\mathbb{C})$ (the algebra of trace-free skew hermitian matrices) with Lie bracket $[X,Y] = XY-YX$ (the matrix commutator). 
For given  $A_{\alpha}: \mathbb{R}^{1+n} \rightarrow g $ we define the curvature by
$$ F_{\alpha \beta} = \partial_{\alpha} A_{\beta} - \partial_{\beta} A_{\alpha} + [A_{\alpha},A_{\beta}] \, , $$
where $\alpha,\beta \in \{0,1,...,n\}$ and $D_{\alpha} = \partial_{\alpha} + [A_{\alpha}, \cdot \,]$ .

Then the Yang-Mills system is given by
\begin{equation}
\label{1'}
D^{\alpha} F_{\alpha \beta}  = 0
\end{equation}
in Minkowski space $\mathbb{R}^{1+n} = \mathbb{R}_t \times \mathbb{R}^n_x$ , where $n \ge 3$, with metric $diag(-1,1,...,1)$. Greek indices run over $\{0,1,...,n\}$, Latin indices over $\{1,...,n\}$, and the usual summation convention is used.  
We use the notation $\partial_{\mu} = \frac{\partial}{\partial x_{\mu}}$, where we write $(x^0,x^1,...,x^n)=(t,x^1,...,x^n)$ and also $\partial_0 = \partial_t$.

Setting $\beta =0$ in (\ref{1'}) we obtain the Gauss-law constraint
\begin{equation}
\nonumber
\partial^j F_{j 0} + [A^j,F_{j0} ]=0 \, .
\end{equation}
The system is gauge invariant. Given a sufficiently smooth function $U: {\mathbb R}^{1+n} \rightarrow \mathcal{G}$ we define the gauge transformation $T$ by $T A_0 = A_0'$ , 
$T(A_1,...,A_n) = (A_1',...,A_n'),$ where
\begin{align*}
A_{\alpha} & \longmapsto A_{\alpha}' = U A_{\alpha} U^{-1} - (\partial_{\alpha} U) U^{-1}  \, . 
\end{align*}
It is  well-known that if  $(A_0,...A_n)$ satisfies (\ref{1'}) so does $(A_0',...,A_n')$.

Hence we may impose a gauge condition. We exlusively study the temporal gauge $A_0=0$.

The Yang-Mills-Higgs system is given by
\begin{align}
\label{1}
D^{\alpha} F_{\alpha \beta} & = [D_{\beta} \phi, \phi ] \\
\label{2}
D^{\alpha} D_{\alpha} \phi & = |\phi|^{N-1} \phi \, .
\end{align}
Setting $\beta =0$ in (\ref{1}) we obtain the Gauss-law constraint
\begin{equation}
\nonumber
\partial^j F_{j 0} = -[A^j,F_{j0}] + [D_0 \phi,\phi ] \, 
\end{equation}
where $\phi: \mathbb{R}^{1+n} \rightarrow g $ .
This system is also gauge invariant. Similarly as above we define the gauge transformation $T$ by $T A_0 = A_0'$ , 
$T(A_1,...,A_n) = (A_1',...,A_n')$ , $T\phi = \phi'$ , where
\begin{align*}
A_{\alpha} & \longmapsto A_{\alpha}' = U A_{\alpha} U^{-1} - (\partial_{\alpha} U) U^{-1} \\
\phi & \longmapsto \phi' = U \phi U^{-1} \, . 
\end{align*}
If $(A_0,...,A_n,\phi)$ satisfies (\ref{1}),(\ref{2}), so does $(A_0',...,A_n',\phi')$.

Some historical remarks:
Concerning the well-posedness problem for the Yang-Mills equation in three space dimensions Klainerman and Machedon \cite{KM1} proved global well-posedness in energy space in the temporal gauge.  Selberg and Tesfahun \cite{ST} proved local well-posedness for finite energy data in Lorenz gauge. This result was improved by Tesfahun \cite{Te} to data without finite energy, namely for $(A(0),(\partial_t A)(0) \in H^s \times H^{s-1}$ with $s > \frac{6}{7}$. Local well-posedness in energy space was given by Oh \cite{O} using a new gauge, namely the Yang-Mills heat flow. He was also able to shows that this solution can be globally extended \cite{O1}. Tao \cite{T1} showed local well-posedness for small data in $H^s \times H^{s-1}$ for $ s > \frac{3}{4}$ in temporal gauge.
In space dimension four where the energy space is critical with respect to scaling Klainerman and Tataru \cite{KT} proved  small data local well-posedness for a closely related model problem in Coulomb gauge for $s>1$. Very recently this result was significantly improved by Krieger and Tataru \cite{KrT}, who were able to show global well-posedness for data with small energy. Sterbenz \cite{St} considered also the four-dimensional case in Lorenz gauge and proved global well-posedness for small data in Besov space $\dot{B}^{1,1} \times \dot(B)^{0,1}$. In high space dimension $n \ge 6$ (and $n$ even)  Krieger and Sterbenz \cite{KrSt} proved global well-posedness for small data in the critical Sobolev space.

Concerning the more general Yang-Mills-Higgs system Eardley and Moncrief \cite{EM},\cite{EM1} proved local and global well-posedness for initial data  $(A(0),(\partial_t A)(0)$ and $(\phi(0),(\partial_t \phi)(0)$) in $H^s \times H^{s-1}$ and $s \ge 2$. In Coulomb gauge global well-posedness in energy space $H^1 \times L^2$ was shown by Keel \cite{K}. Recently Tesfahun \cite{Te1} considered the problem in Lorenz gauge and obtained local well-posedness in energy space.

We now study the Yang-Mills equation and also the Yang-Mills-Higgs system in arbitrary space dimension $n \ge 3$ in temporal gauge for low regularity data, which in three space dimension not necessarily have finite energy and which fulfill a smallness assumption, which reads in the Yang-Mills-Higgs case as follows
$$ \|A(0)\|_{H^s} + \|(\partial_t A)(0)\|_{H^{s-1}} + \|\phi(0)\|_{H^s} + \|(\partial_t \phi)(0)\|_{H^{s-1}} < \epsilon $$
with a sufficiently small $\epsilon > 0$ , under the assumption $s> \frac{3}{4}$ for $n=3$ and $s > \frac{n}{2}-\frac{5}{8}-\frac{5}{8(2n-1)}$ in general dimension $n\ge3$. We obtain a solution which satisfies $A,\phi \in C^0([0,1],H^s) \cap C^1([0,1],H^{s-1})$.
A corresponding result holds for the Yang-Mills equation.
Uniqueness holds in a certain subspace of Bourgain-Klainerman-Machedon type.
The basis for our results is Tao's paper \cite{T1}. We carry over his three-dimensional result for the Yang-Mills equation to the more general Yang-Mills-Higgs equations and to arbitrary dimensions $n\ge 3$. The result relies on the null structure of all the critical bilinear terms. We review this null structure which was partly detected already by Klainerman-Machedon in the Yang-Mills case \cite{KM1} and by Tesfahun \cite{Te} for Yang-Mills-Higgs in the situation of the Lorenz gauge. The necessary estimates for the nonlinear terms in spaces of $X^{s,b}$-type in the (3+1)-dimensional case then reduce essentially to Tao's result \cite{T1}. One of these estimates is responsible for the small data assumption. Because these local well-posedness results (Prop. \ref{Prop'}) and (Prop. \ref{Prop}) can initially only be shown under the condition that the curl-free part $A^{cf}$ of $A$ (as defined below) vanishes for $t=0$ we have to show that this assumption can be removed by a suitable gauge transformation (Lemma \ref{Lemma}) which preserves the regularity of the solution. This uses an idea of Keel and Tao \cite{T1}.

Our main results read as follows:
\begin{theorem}
\label{Theorem1'}
Let $ n \ge 3$ , $s > \frac{n}{2}-\frac{5}{8}-\frac{5}{8(2n-1)}$ . Let $a \in H^s({\mathbb R}^n)$ , $a' \in H^{s-1}({\mathbb R}^n)$  be given, where $a=(a_1,...,a_n)$ , $a' =(a_1',...,a_n')$ ,  satisfying the compatability condition $\partial^j a_j' = - [a^j,a_j']$. Assume
$\, \|a\|_{H^s} + \|a'\|_{H^{s-1}}  \le \epsilon \, , $
where $\epsilon > 0$ is sufficiently small. Then the Yang-Mills equation (\ref{1'}) in temporal gauge $A_0=0$ with initial conditions
$ A(0)=a \, , \, (\partial_t A)(0) = a'  \, ,$
where $A=(A_1,...,A_n)$,
has a unique local solution $A= A_+^{df} + A_-^{df} +A^{cf}$ , where
$$ A^{df}_{\pm} \in X^{s,\frac{3}{4}+}_{\pm}[0,1] \, , \, A^{cf} \in X^{s+\alpha,\frac{1}{2}+}_{\tau=0}[0,1] \, , \, \partial_t A^{cf} \in C^0([0,1],H^{s-1}) \, .$$
These spaces are defined below and $\alpha = \frac{3n+1}{8(2n-1)}$. This solution fulfills
$$ A \in C^0([0,1],H^s({\mathbb R}^n)) \cap C^1([0,1],H^{s-1}({\mathbb R}^n)) \, . $$
\end{theorem}
{\bf Remark:} In the (3+1)-dimensional case we assume $ s > \frac{3}{4} $ and $\alpha = \frac{1}{4}$ , so that data without finite energy are admissible. This is Tao's result \cite{T1}.
\begin{theorem}
\label{Theorem1}
Let $ n \ge 3$ , $s > \frac{n}{2}-\frac{5}{8}-\frac{5}{8(2n-1)}$ , and  $2 \le N < 1+\frac{7}{4(\frac{n}{2}-s)}$ , if $s < \frac{n}{2}$ , and $N< \infty$ , if $s \ge \frac{n}{2}$ . Here $N$ is an odd integer, or $N \in\mathbb{N}$ with $N > s$. Let $a \in H^s({\mathbb R}^n)$ , $a' \in H^{s-1}({\mathbb R}^n)$ , $\phi_0 \in H^s({\mathbb R}^3)$ , $\phi_1 \in H^{s-1}({\mathbb R}^3)$ be given, where $a=(a_1,...,a_n)$ , $a' =(a_1',...,a_n')$ ,  satisfying $\partial^j a_j' = -\partial^j a_j'- [\phi_1,\phi_0]$. Assume
$$ \|a\|_{H^s} + \|a'\|_{H^{s-1}} + \|\phi_0\|_{H^s} + \|\phi_1\|_{H^{s-1}} \le \epsilon \, , $$
where $\epsilon > 0$ is sufficiently small. Then the Yang-Mills-Higgs equations (\ref{1}) , (\ref{2}) in temporal gauge $A_0=0$ with initial conditions
$$ A(0)=a \, , \, (\partial_t A)(0) = a' \, , \, \phi(0)=\phi_0 \, , \, (\partial_t \phi)(0) = \phi_1 \, ,$$
where $A=(A_1,...,A_n)$,
has a unique local solution  $A= A_+^{df} + A_-^{df} +A^{cf}$ and $\phi = \phi_+ + \phi_-$ , where
$$ A^{df}_{\pm} \in X^{s,\frac{3}{4}+}_{\pm}[0,1]  ,  A^{cf} \in X^{s+\alpha,\frac{1}{2}+}_{\tau=0}[0,1]  ,  \partial_t A^{cf} \in C^0([0,1],H^{s-1})  , 
  \phi_{\pm} \in  X^{s,\frac{3}{4}+}_{\pm}[0,1] \, , $$
where these spaces are defined below and $\alpha = \frac{3n+1}{8(2n-1)}$. This solution fulfills
$$ A \, , \, \phi \in C^0([0,1],H^s({\mathbb R}^n)) \cap C^1([0,1],H^{s-1}({\mathbb R}^n)) \, . $$
\end{theorem}
Remark: The assumption $N>s$ or $N$ odd ensures that the function $f(s) = |s|^{N-1} s$ for $s\in \mathbb{R}$ is smooth enough at the origin.

We denote the Fourier transform with respect to space and time and with respect to space by $\,\,\widehat{}\,\,$ and ${\mathcal F}$, respectively. The operator
$|\nabla|^{\alpha}$ is defined by $({\mathcal F}(|\nabla|^{\alpha} f))$ $(\xi) = |\xi|^{\alpha} ({\mathcal F}f)(\xi)$ and similarly $ \langle \nabla \rangle^{\alpha}$. $\Box = \partial_t^2 - \Delta$ is the d'Alembert operator.\\
$a+ := a + \epsilon$ for a sufficiently small $\epsilon >0$ , so that $a<a+<a++$ , and similarly $a--<a-<a$ , and $\langle \cdot \rangle := (1+|\cdot|^2)^{\frac{1}{2}}$ .

The standard spaces $X^{s,b}_{\pm}$ of Bourgain-Klainerman-Machedon type belonging to the half waves are the completion of the Schwarz space  $\mathcal{S}({\mathbb R}^4)$ with respect to the norm
$$ \|u\|_{X^{s,b}_{\pm}} = \| \langle \xi \rangle^s \langle  \tau \mp |\xi| \rangle^b \widehat{u}(\tau,\xi) \|_{L^2_{\tau \xi}} \, . $$ 
Similarly we define the wave-Sobolev spaces $X^{s,b}_{|\tau|=|\xi|}$ with norm
$$ \|u\|_{X^{s,b}_{|\tau|=|\xi|}} = \| \langle \xi \rangle^s \langle  |\tau| - |\xi| \rangle^b \widehat{u}(\tau,\xi) \|_{L^2_{\tau \xi}}  $$ and also $X^{s,b}_{\tau =0}$ with norm 
$$\|u\|_{X^{s,b}_{\tau=0}} = \| \langle \xi \rangle^s \langle  \tau  \rangle^b \widehat{u}(\tau,\xi) \|_{L^2_{\tau \xi}} \, .$$
We also define $X^{s,b}_{\pm}[0,T]$ as the space of the restrictions of functions in $X^{s,b}_{\pm}$ to $[0,T] \times \mathbb{R}^3$ and similarly $X^{s,b}_{|\tau| = |\xi|}[0,T]$ and $X^{s,b}_{\tau =0}[0,T]$. We frequently use the estimates $\|u\|_{X^{s,b}_{\pm}} \le \|u\|_{X^{s,b}_{|\tau|=|\xi|}}$ for $b \le 0$ and the reverse estimate for $b \ge 0$. \\

\section{Reformulation of the problem and null structure}

In temporal gauge $A_0=0$ the system (\ref{1'}) is equivalent to
\begin{align*}
\partial_t \,div \, A & = - [A_i,\partial_t A^i] \\
\Box A_j & = \partial_j \, div A - [div A,A_j] - 2[A^i,\partial_i A_j] + [A^i,\partial_j A_i] - [A^i,[A_i,A_j]] 
\end{align*}
and the Gauss constraint  reduces to
$$\partial^j \partial_t A_j = - [A_j,\partial_t A^j] \, . $$

Similarly in temporal gauge $A_0=0$ the system (\ref{1}),(\ref{2}) is equivalent to
\begin{align*}
\partial_t \,div \, A & = -[\partial_t \phi , \phi] - [A_i,\partial_t A^i] \\
\Box A_j & = \partial_j \, div A - [div A,A_j] - 2[A^i,\partial_i A_j] + [A^i,\partial_j A_i] - [\phi,\partial_j \phi] - [A^i,[A_i,A_j]] \\
&\hspace{1em} - [\phi,[A_j,\phi]] \\
\Box \phi & = -[div A,\phi] - 2[A_i,\partial^i \phi] - [A^i,[A_i,\phi]] + |\phi|^{N-1} \phi
\end{align*}
and the Gauss constraint  reduces to
$$\partial^j \partial_t A_j = -[A_j,\partial_t A^j]- [\partial_t \phi,\phi] \, . $$
We decompose $A$ into its divergence-free part $A^{df}$ and its curl-free part $A^{cf}$ :
$$
A = A^{df} + A^{cf} \, ,
$$
where
$$
 A_j^{df} = (PA)_j := R^k(R_j A_k - R_k A_j) \quad , \quad A_j^{cf} = - R_j R_k A^k \, .
 $$
Here $P$ denotes the Leray projection onto the divergence-free part, and $R_j := |\nabla|^{-1} \partial_j$ is the Riesz transform. 

Then we obtain the following system which is equivalent to (\ref{1'}): 
\begin{align}
\label{3'}
\partial_t A^{cf} &=   (-\Delta)^{-1} \nabla [A_i,\partial_t A^i] \\
\label{4'}
\Box A^{df} & = -P [div \, A^{cf},A] - 2 P[A^i,\partial_i A] + P[A^i,\nabla A_i]  - P[A^i,[A_i,A]] 
\end{align}

Similarly the following system is equivalent to (\ref{1}),(\ref{2}):
\begin{align}
\label{3}
\partial_t A^{cf} &= - (-\Delta)^{-1} \nabla [\partial_t \phi,\phi] + (-\Delta)^{-1} \nabla [A_i,\partial_t A^i] \\
\nonumber
\Box A^{df} & = -P [div \, A^{cf},A] - 2P[A^i,\partial_i A] + P[A^i,\nabla A_i] -P[\phi,\nabla \phi] - P[A^i,[A_i,A]] \\
\label{4}
& \hspace{1em}- P[\phi,[A,\phi]] \\
\label{5}
\Box \phi & = -[div \, A^{cf},\phi] - 2[A_i,\partial^i \phi] - [A^i,[A_i,\phi]] + |\phi|^{N-1} \phi \, .
\end{align}
We now show that all the critical terms in (\ref{4'}), (\ref{4}) and (\ref{5}), namely the quadratic terms which contain only $A^{df}$ or $\phi$ have null structure. Those quadratic terms which contain $A^{cf}$ are less critical, because $A^{cf}$ is shown to be more regular than $A^{df}$, and the cubic terms are also less critical, because they contain no derivatives. The only critical term in (\ref{5}) is $[A^{df}_i,\partial^i \phi]$. We easily calculate
\begin{align}
\nonumber
&[A^{df}_i,\partial^i A^{df}] = [R^k(R_i A_k - R_k A_i),\partial^i A^{df}] \\
\nonumber
&= \frac{1}{2} \big([R^k(R_i A_k - R_k A_i),\partial^i A^{df}] + [R^i(R_k A_i - R_i A_k),\partial^k A^{df}]\big) \\
\nonumber
&=\frac{1}{2} \big([R^k(R_i A_k - R_k A_i),\partial^i A^{df}] - [R^i(R_i A_k - R_k A_i),\partial^k A^{df}]\big) \\
\label{50}
&= \frac{1}{2} Q^{ik} [ |\nabla|^{-1}(R_i A_k - R_k A_i),A^{df}]
\end{align}
where
$$ Q_{ij}[u,v]  := [\partial_i u,\partial_jv] - [\partial_j u,\partial_i v] = Q_{ij}(u,v) + Q_{ji}(v,u) $$
with the standard null form
$$ Q_{ij}(u,v) := \partial_i u \partial_j v - \partial_j u \partial_i v \, . $$ 
Thus, ignoring $P$, which is a bounded operator we obtain
\begin{equation}
\label{N2}
P[A_i^{df},\partial^i A^{df}] \sim \sum Q_{ik}[|\nabla|^{-1} A^{df},A^{df}] \, ,
\end{equation}
and similarly
\begin{equation}
\label{N2'}
P[A_i^{df},\partial^i \phi] \sim \sum Q_{ik}[|\nabla|^{-1} A^{df},\phi] \, .
\end{equation}
Moreover
\begin{align*}
(\phi \nabla \phi')^{df}_j & = R^k(R_j(\phi \partial_k \phi') - R_k(\phi \partial_j \phi')) \\
& = |\nabla|^{-2} \partial^k(\partial_j (\phi \partial_k \phi') - \partial_k (\phi \partial_j \phi'))) \\
& = |\nabla|^{-2} \partial^k(\partial_j \phi \partial_k \phi' - \partial_k \phi \partial_j \phi') \\
& = |\nabla|^{-2} \partial^k Q_{jk}(\phi,\phi')
\end{align*}
so that
\begin{equation}
\label{N3}
P[\phi,\nabla \phi] \sim \sum |\nabla|^{-1} Q_{jk} [\phi,\phi] \, ,
\end{equation}
and
\begin{equation}
\label{N3'}
P[A^{df}_i,\nabla A^{df}_i] \sim \sum  |\nabla|^{-1}Q_{jk} [A^{df},A^{df}] \, ,
\end{equation}
All the other quadratic terms contain at least one factor $A^{cf}$.

Defining
\begin{align*}
\phi_{\pm} = \frac{1}{2}(\phi \mp i \langle \nabla \rangle^{-1} \partial_t \phi)&
 \Longleftrightarrow \phi=\phi_+ + \phi_- \, , \, \partial_t \phi = i \langle \nabla \rangle (\phi_+ - \phi_-) \\
 A^{df}_{\pm} = \frac{1}{2}(A^{df} \mp i \langle \nabla \rangle^{-1} \partial_t A^{df}) & \Longleftrightarrow A^{df} = A^{df}_+ + A_-^{df} \, , \, \partial_t A^{df} = i \langle \nabla \rangle(A^{df}_+ - A^{df}_-)
 \end{align*}
 we can rewrite (\ref{3'}),(\ref{4'}) as
 \begin{align}
 \label{6'}
 \partial_t A^{cf} &=  (-\Delta)^{-1} \nabla [A_i,\partial_t A^i] \\
 \label{7'}
(i \partial_t \pm \langle \nabla \rangle)A_{\pm} ^{df} & = \mp 2^{-1} \langle \nabla \rangle^{-1} ( R.H.S. \, of \, (\ref{4'}) - A^{df}) \, .
\end{align}
with initial data
\begin{align}
\label{1.15*'}
A^{df}_\pm(0) & = \frac{1}{2}(A^{df}(0) \mp i^{-1} \langle \nabla \rangle^{-1} \partial_t A^{df}(0) \, .
\end{align}
 Similarly we can rewrite (\ref{3}),(\ref{4}),(\ref{5}) as
 \begin{align}
 \label{6}
 \partial_t A^{cf} &=  (-\Delta)^{-1} \nabla [\partial_t \phi,\phi] + (-\Delta)^{-1} \nabla [A_i,\partial_t A^i] \\
 \label{7}
(i \partial_t \mp \langle \nabla \rangle)A_{\pm} ^{df} & = \mp 2^{-1} \langle \nabla \rangle^{-1} ( R.H.S. \, of \, (\ref{4}) - A^{df}) \\
\label{8}
(i \partial_t \mp \langle \nabla \rangle) \phi_{\pm} &= \mp 2^{-1} \langle \nabla \rangle^{-1}( R.H.S. \, of \, (\ref{5}) - \phi) \, .
\end{align}
The initial data are transformed as follows:
\begin{align}
\label{1.14*}
\phi_{\pm}(0) &= \frac{1}{2}(\phi(0) \mp i^{-1} \langle \nabla \rangle^{-1} \partial_t \phi(0)) \\
\label{1.15*}
A^{df}_\pm(0) & = \frac{1}{2}(A^{df}(0) \mp i^{-1} \langle \nabla \rangle^{-1} \partial_t A^{df}(0) \, .
\end{align}

\section{The preliminary local well-posedness results}

We now state and prove preliminary local well-posedness of (\ref{3'}),(\ref{4'}) as well as (\ref{3}),(\ref{4}),(\ref{5}), for which it is essential to have data for $A$ with vanishing curl-free part. 

\begin{prop}
\label{Prop'}
For space dimension $n \ge 3$ assume $s>\frac{n}{2}- \frac{5}{8}-\frac{5}{8(2n-1)}$ and $\alpha = \frac{3n+1}{8(2n-1)}$ .
Let
$a^{df} = (a_1^{df},...,a_n^{df}) \in H^s$ , $a'^{df}= ({a'}_1^{df},...,{a'}_n^{df}) \in H^{s-1}$ be given with
$$ \sum_j \|a_j^{df}\|_{H^s} + \sum_j \|{a'}_j^{df}\|_{H^{s-1}} \le \epsilon_0 \, ,$$
where $\epsilon_0 >0$ is sufficiently small. Then the system (\ref{3'}),(\ref{4'}) with initial conditions
$$ A^{df}(0)=a^{df} \, , \, (\partial_t A^{df})(0) = {a'}^{df} \, , \, A^{cf}(0) = 0 \, , $$
has a unique local solution 
$$ A= A^{df}_+ + A^{df}_- + A^{cf} \, , $$
where 
$$  A^{df}_{\pm} \in X^{s,\frac{3}{4}+}_{\pm}[0,1] \, , \, A^{cf} \in X^{s+\alpha,\frac{1}{2}+}_{\tau=0}[0,1] \, , \, \partial_t A^{cf} \in C^0([0,1],H^{s-1}) \, .$$
Uniqueness holds (of course) for not necessarily vanishing initial data $A^{cf}(0) = a^{cf}$. The solution satisfies
$$ A \in C^0([0,1],H^s) \cap C^1([0,1],H^{s-1}) \, . $$
\end{prop}

\begin{prop}
\label{Prop}
For space dimension $n \ge 3$ assume $s>\frac{n}{2}- \frac{5}{8}-\frac{5}{8(2n-1)}$ and $\alpha = \frac{3n+1}{8(2n-1)}$ . Assume $2 \le N < 1 + \frac{7}{4(\frac{n}{2}-s)}$ , if $s < \frac{n}{2}$ , and $2 \le N < \infty$,  if $ s \ge \frac{n}{2}$. Here $N$ is an odd integer, or $N \in\mathbb{N}$ with $N > s$.
Let
$a^{df} = (a_1^{df},...,a_n^{df}) \in H^s$ , $a'^{df}= ({a'}_1^{df},...,{a'}_n^{df}) \in H^{s-1}$ , $\phi_0 \in H^s$,  $\phi_1 \in H^{s-1}$ 
be given with
$$ \sum_j \|a_j^{df}\|_{H^s} + \sum_j \|{a'}_j^{df}\|_{H^{s-1}} + \|\phi_0\|_{H^s} + \|\phi_1\|_{H^{s-1}} \le \epsilon_0 \, ,$$
where $\epsilon_0 >0$ is sufficiently small. Then the system (\ref{3}),(\ref{4}),(\ref{5}) with initial conditions
$$ \phi(0) = \phi_0 \, , \, (\partial_t \phi)(0) = \phi_1 \, , \, A^{df}(0)=a^{df} \, , \, (\partial_t A^{df})(0) = {a'}^{df} \, , \, A^{cf}(0) = 0 \, , $$
has a unique local solution 
$$ \phi= \phi_+ + \phi_- \quad , \quad A= A^{df}_+ + A^{df}_- + A^{cf} \, , $$
where 
$$ \phi_{\pm} \in X^{s,\frac{3}{4}+}_{\pm}[0,1]  ,  A^{df}_{\pm} \in X^{s,\frac{3}{4}+}_{\pm}[0,1]  ,  A^{cf} \in X^{s+\alpha,\frac{1}{2}+}_{\tau=0}[0,1]  ,  \partial_t A^{cf} \in C^0([0,1],H^{s-1}) \, . $$
Uniqueness holds (of course) for not necessarily vanishing initial data $A^{cf}(0) = a^{cf}$. The solution satisfies
$$ A,\phi \in C^0([0,1],H^s) \cap C^1([0,1],H^{s-1}) \, . $$
\end{prop}
\vspace{1cm}
Fundamental for their proof are the following estimates.

\begin{prop}
\label{Prop.2}
Let $n \ge 2$.
\begin{enumerate}
\item For $2 < q \le \infty $ , $ 2 \le r < \infty$ , $ \frac{2}{q} = (n-1)(\frac{1}{2}-\frac{1}{r})$ , $ \mu = n(\frac{1}{2}-\frac{1}{r})-\frac{1}{q}$ the following estimate holds
\begin{equation}
\label{15}
\|u\|_{L^q_t L^r_x} \lesssim \|u\|_{X^{\mu,\frac{1}{2}+}_{|\tau|=|\xi|}} \, . 
\end{equation}
\item For $k \ge 0$ , $ p < \infty$ and $ \frac{n-1}{2(n+1)} \ge \frac{1}{p} \ge \frac{n-1}{2(n+1)} - \frac{k}{n}$ the following estimate holds:
\begin{equation}
\label{Tao}
\|u\|_{ L^p_x L^2_t} \lesssim \|u\|_{X^{k+\frac{n-1}{2(n+1)},\frac{1}{2}+}_{|\tau|=|\xi|}} \, . 
\end{equation}
\end{enumerate}
\end{prop}
\begin{proof}
(\ref{15}) is the Strichartz type estimate, which can be found for e.g. in \cite{GV}, Prop. 2.1, combined with the transfer principle.

Concerning (\ref{Tao}) we use \cite{KMBT}, Thm. B.2:
$$ \|\mathcal{F}_t u \|_{L^2_{\tau} L_x^{\frac{2(n+1)}{n-1}}} \lesssim \|u_0\|_{\dot{H}^{\frac{n-1}{2(n+1)}}} \, , $$
if $u=e^{it |\nabla|} u_0$ and $\mathcal{F}_t$ denotes the Fourier transform with respect to time. This immediately implies by Plancherel, Minkowski's inequality and Sobolev's embedding theorem
$$\|u\|_{L^p_x L^2_t} = \|\mathcal{F}_t u \|_{L^p_x L^2_\tau} \le \|\mathcal{F}_t u \|_{L^2_{\tau} L^p_x} \lesssim \|\mathcal{F}_t u \|_{L^2_{\tau} H^{k,\frac {2(n+1)}{n-1}}_x} \lesssim \|u_0\|_{H^{k+\frac{n-1}{2(n+1)}}} \, . $$
The transfer principle implies (\ref{Tao}).
\end{proof}

\begin{proof}[Proof of Prop. \ref{Prop} and Prop. \ref{Prop'}] 
We use the system (\ref{6'}),(\ref{7'}) (instead of (\ref{3'}),(\ref{4'})) and (\ref{6}),(\ref{7}),(\ref{8}) (instead of (\ref{3}),(\ref{4}),(\ref{5})) with initial conditions (\ref{1.15*'}) and (\ref{1.14*}),(\ref{1.15*}). We want to use a contraction argument for  $A_{\pm}^{df} \in X_{\pm}^{s,\frac{3}{4}+\epsilon}[0,1] \, , \, A^{cf} \in X^{s+\alpha,\frac{1}{2}+\epsilon}_{\tau=0}$ $[0,1]$ , $\partial_t A^{cf} \in C^0([0,1],H^{s-1}$ ), and in the Yang-Mills-Higgs case in addition for $\phi \in X_{\pm}^{s,\frac{3}{4}+ \epsilon}[0,1]$ . 
Provided that our small data assumption holds this can be reduced  by well-known arguments to suitable multilinear estimates of the right hand sides of these equations. 
For (\ref{7'}) e.g. we make use of the following well-known estimate:
$$
\|A^{df}_{\pm}\|_{X^{l,b}_{\pm}[0,1]} \lesssim \|A^{df}_{\pm}(0)\|_{H^l} +  \| R.H.S. \, of \, (\ref{7'}) \|_{X^{l,b-1}_{\pm}[0,1]} \, , $$
which holds for $l\in{\mathbb R}$  , $\frac{1}{2} < b \le 1$ .

Thus the local existence and uniqueness can be reduced to the following estimates.

In order to control $A^{cf}$ we need
\begin{align}
\label{16}
\| |\nabla|^{-1} (\phi_1 \partial_t \phi_2)\|_{X^{s+\alpha,-\frac{1}{2}+\epsilon+}_{\tau=0}} &\lesssim \|\phi_1\|_{X^{s,\frac{3}{4}+\epsilon}_{|\tau|=|\xi|}} \|\phi_2\|_{X^{s,\frac{3}{4}+\epsilon}_{|\tau|=|\xi|}} \\
\label{17}
\| |\nabla|^{-1} (\phi_1 \partial_t \phi_2)\|_{X^{s+\alpha,-\frac{1}{2}+2\epsilon-}_{\tau=0}} &\lesssim \|\phi_1\|_{X^{s+\alpha,\frac{1}{2}+\epsilon}_{\tau=0}} \|\phi_2\|_{X^{s+\alpha,\frac{1}{2}+\epsilon}_{\tau=0}} \\
\label{18}
\| |\nabla|^{-1} (\phi_1 \partial_t \phi_2)\|_{X^{s+\alpha,-\frac{1}{2}+\epsilon}_{\tau=0}} &+ \| |\nabla|^{-1} (\phi_2 \partial_t \phi_1)\|_{X^{s+\alpha,-\frac{1}{2}+\epsilon}_{\tau=0}} \\ \nonumber
&\lesssim \|\phi_1\|_{X^{s+\alpha,\frac{1}{2}+\epsilon}_{\tau=0}} \|\phi_2\|_{X^{s,\frac{3}{4}+\epsilon}_{|\tau|=|\xi|}} \, .
\end{align}
In order to control $\partial_t A^{cf}$ we need
\begin{align}
\label{19}
\| |\nabla|^{-1} (A_1 \partial_t A_2)\|_{C^0(H^{s-1})} \lesssim &
(\|A_1^{cf}\|_{X^{s+\alpha,\frac{1}{2}+}_{\tau=0}} + \sum_{\pm} 
\|A^{df}_{1\pm}\|_{X^{s,\frac{1}{2}+}_{\pm}})\\
\nonumber
&(\|\partial_t 
A^{cf}_2\|_{C^0(H^{s-1})} +  \sum_{\pm} \|A^{df}_{2\pm}\|_{X^{s,\frac{1}{2}+}_{\pm}}) \, .
\end{align}
The estimate for $A^{df}$ and $\phi$ by use of (\ref{N2}),(\ref{N2'}),(\ref{N3}),(\ref{N3'}) reduces to
\begin{align}
\nonumber
&\|Q_{ij}(|\nabla|^{-1}\phi_1,\phi_2)\|_{X^{s-1,-\frac{1}{4}+2\epsilon}_{|\tau|=|\xi|}} + \|\nabla^{-1}Q_{ij}(\phi_1,\phi_2)\|_{X^{s-1,-\frac{1}{4}+2\epsilon}_{|\tau|=|\xi|}} \\
\label{28}
&\hspace{1em}\lesssim \|\phi_1\|_{X^{s,\frac{3}{4}+\epsilon}_{|\tau|=|\xi|}} \|\phi_2\|_{X^{s,\frac{3}{4}+\epsilon}_{|\tau|=|\xi|}} \, .
\end{align}
For the proof of (\ref{28}) we refer to \cite{T}, Prop. 9.2 (slightly modified), which is given under the assumption $s > \frac{n}{2}-\frac{3}{4}$. This assumption is weaker than our assumption, if $n\ge 4$, and they coincide for $n=3$.

Moreover for the terms $P[div \,A^{cf},A]$ , $P[A^i,\partial_i A]$ and $P[A^i,\partial_j A_i]$ we need
\begin{equation}
\label{29}
\| \nabla A^{cf} A^{df} \|_{X^{s-1,-\frac{1}{4}+2\epsilon}_{|\tau|=|\xi|}} +
\| A^{cf} \nabla A^{df} \|_{X^{s-1,-\frac{1}{4}+2\epsilon}_{|\tau|=|\xi|}} \lesssim \|A^{cf}\|_{X^{s+\alpha,\frac{1}{2}+\epsilon}_{\tau =0}}  \|A^{df}\|_{X^{s,\frac{3}{4}+\epsilon}_{|\tau|=|\xi|}} 
\end{equation}
 and
\begin{equation}
\label{30}
\| \nabla A^{cf} A^{cf} \|_{X^{s-1,-\frac{1}{4}+2\epsilon}_{|\tau|=|\xi|}}  \lesssim \|A^{cf}\|_{X^{s+\alpha,\frac{1}{2}+\epsilon}_{\tau =0}}^2  \, .
\end{equation}
All the cubic terms are estimated by
\begin{equation}
\label{31}
\| A_1 A_2 A_3 \|_{X^{s-1,-\frac{1}{4}+2\epsilon}_{|\tau|=|\xi|}} \lesssim \prod_{i=1}^3 \min(\|A_i\|_{X^{s,\frac{3}{4}+\epsilon}_{|\tau|=|\xi|}},\|A_i\|_{X^{s+\alpha,\frac{1}{2}+\epsilon}_{\tau =0}} ) \, .
\end{equation}
Remark that in (\ref{19}), (\ref{29}) and (\ref{31}) $A^{df}$ may be replaced by $\phi$ . \\
For the Yang-Mills-Higgs system we additionally need
\begin{equation}
\label{32}
\| |\phi|^{N-1} \phi \|_{X^{s-1,-\frac{1}{4}+2\epsilon}_{|\tau|=|\xi|}} \lesssim \|\phi\|^N_{X^{s,\frac{3}{4}+\epsilon}_{|\tau|=|\xi|}} \, .
\end{equation}
All these estimates up to (\ref{19}) and (\ref{32}) have been essentially given by Tao \cite{T1} for the Yang-Mills case in space dimension $n=3$.  We remark that it is especially (\ref{18}) which prevents a large data result, because it seems to be difficult to replace $X^{s+\alpha,-\frac{1}{2}+\epsilon}_{\tau=0}$ by $X^{s+\alpha,-\frac{1}{2}+\epsilon+}_{\tau=0}$ on the left hand side.\\
{\bf Proof of (\ref{17}).}
As usual the regularity of $|\nabla|^{-1}$ is harmless in dimension $n\ge 3$ (\cite{T}, Cor. 8.2) and it can be replaced by $\langle \nabla \rangle^{-1}$. Taking care of the time derivative we reduce to
\begin{align*}
\big|\int \int u_1 u_2 u_3 dx dt\big| \lesssim \|u_1\|_{X^{s+\alpha,\frac{1}{2}+\epsilon}_{\tau =0}}
\|u_2\|_{X^{s+\alpha,-\frac{1}{2}+\epsilon}_{\tau =0}}
\|u_3\|_{X^{1-(\alpha +s),\frac{1}{2}-2\epsilon+}_{\tau =0}} \, ,
\end{align*}
which follows from Sobolev's multiplication rule, because under our assumption on $s$ and the choice of $\alpha$ we obtain
$ 2(s+\alpha)+1-(\alpha +s) > \frac{n}{2}$ , as one easily calculates. \\
{\bf Proof of (\ref{18}).}
a. If $\widehat{\phi}$ is supported in  $ ||\tau|-|\xi|| \gtrsim |\xi| $ , we obtain
$$ \|\phi\|_{X^{s+\alpha,\frac{1}{2}+\epsilon}_{\tau=0}} \lesssim \|\phi\|_{X^{s,\frac{3}{4}+\epsilon}_{|\tau|=|\xi|}} \,, $$
when we remark that $\alpha \le \frac{1}{4}$ for $n\ge 3$ .
Thus (\ref{18}) follows from (\ref{17}).\\
b. It remains to show
$$ \big|\int\int (uv_t w + uvw_t) dxdt \big| \lesssim 
\|u\|_{X^{1-\alpha-s,\frac{1}{2}-\epsilon}_{\tau =0}}
\|w\|_{X^{s,\frac{3}{4}+\epsilon}_{|\tau| =|\xi|}}
\|v\|_{X^{s+\alpha-\epsilon,\frac{1}{2}+\epsilon}_{\tau =0}} \, $$
whenever $\widehat w$ is supported in $||\tau|-|\xi|| \ll |\xi|$.
This is equivalent to
$$ \int_* m(\xi_1,\xi_2,\xi_3,\tau_1,\tau_2,\tau_3) \prod_{i=1}^3 \widehat{u}_i(\xi_i,\tau_i) d\xi d\tau \lesssim \prod_{i=1}^3 \|u_i\|_{L^2_{xt}} \, $$
where $d\xi = d\xi_1 d\xi_2 d\xi_3$ , $d\tau = d\tau_1 d\tau_2 d\tau_n$ and * denotes integration over $\sum_{i=1}^3 \xi_i = \sum_{i=1}^3 \tau_i = 0$. The Fourier transforms are nonnegative without loss of generality. Here
$$ m= \frac{(|\tau_2|+|\tau_3|) \chi_{||\tau_3|-|\xi_3|| \ll |\xi_3|}}{\langle \xi_1 \rangle^{1-\alpha-s} \langle \tau_1 \rangle^{\frac{1}{2}-\epsilon} \langle \xi_2 \rangle^{s+\alpha-\epsilon} \langle \tau_2 \rangle^{\frac{1}{2}+\epsilon} \langle \xi_3 \rangle^s \langle |\tau_3|-|\xi_3|\rangle^{\frac{3}{4}+\epsilon}} \, . $$
Since $\langle \tau_3 \rangle \sim \langle \xi_3 \rangle$ and $\tau_1+\tau_2+\tau_3=0$ we have 
\begin{equation}
\label{N4'}
|\tau_2| + |\tau_3| \lesssim \langle \tau_1 \rangle^{\frac{1}{2}-\epsilon} \langle \tau_2 \rangle^{\frac{1}{2}+\epsilon} +\langle \tau_1 \rangle^{\frac{1}{2}-\epsilon} \langle \xi_3 \rangle^{\frac{1}{2}+\epsilon} +\langle \tau_2 \rangle^{\frac{1}{2}+\epsilon} \langle \xi_3 \rangle^{\frac{1}{2}-\epsilon} , 
\end{equation}
so that concerning the first term on the right hand side of (\ref{N4'}) we have to show
$$\big|\int\int uvw dx dt\big| \lesssim \|u\|_{X^{1-\alpha-s,0}_{\tau=0}} \|v\|_{X^{s+\alpha-\epsilon,0}_{\tau=0}} \|w\|_{X^{s,\frac{3}{4}+\epsilon}_{|\tau|=|\xi|}} \ , $$
which easily follows from Sobolev's multiplication rule, because $s> \frac{n}{2}-1$.\\
Concerning the second term on the right hand side of (\ref{N4'}) we use $\langle \xi_1 \rangle^{s-1+\alpha} \lesssim \langle \xi_2 \rangle^{s-1+\alpha} + \langle \xi_3 \rangle^{s-1+\alpha}$, so that we reduce to
\begin{equation}
\label{51}
\big|\int\int uvw dx dt\big|  
 \lesssim\|u\|_{X^{0,0}_{\tau=0}} \|v\|_{X^{1-\epsilon,\frac{1}{2}+\epsilon}_{\tau=0}} \|w\|_{X^{s-\frac{1}{2}-\epsilon,\frac{3}{4}+\epsilon}_{|\tau|=|\xi|}} 
\end{equation}
and 
\begin{equation}
\label{52}
\big|\int\int uvw dx dt\big|  
 \lesssim\|u\|_{X^{0,0}_{\tau=0}}
\|v\|_{X^{s+\alpha-\epsilon,\frac{1}{2}+\epsilon}_{\tau=0}} \|w\|_{X^{\frac{1}{2}-\alpha-\epsilon,\frac{3}{4}+\epsilon}_{|\tau|=|\xi|}} \,.
\end{equation}
To obtain (\ref{51}) in the case $n \ge 4$ we estimate as follows:
$$ \big| \int\int uvw dx dt \big| \le \|u\|_{L^2_xL^2_t} \|v\|_{L^{\frac{2n}{n-2+2\epsilon}}_x L^{\infty}_t} \|w\|_{L^{\frac{n}{1-\epsilon}}_xL^2_t} \, . $$
We use (\ref{Tao}) with $p=\frac{n}{1-\epsilon}$ and $k= n(\frac{n-1}{2(n+1)}-\frac{1}{p})$, so that one easily checks that
$$ k + \frac{n-1}{2(n+1)} = n(\frac{n-1}{2(n+1)} - \frac{1-\epsilon}{n}) + \frac{n-1}{2(n+1)} < \frac{n}{2} - \frac{5}{4} < s - \frac{1}{2} - \epsilon \, . $$
Thus
$$ \|w\|_{L^{\frac{n}{1-\epsilon}}_x L^2_t} \lesssim \|w\|_{X^{k+\frac{n-1}{2(n+1)},\frac{1}{2}+}_{|\tau|=|\xi|}} \le \|w\|_{X^{s-\frac{1}{2}-\epsilon,\frac{1}{2}+}_{|\tau|=|\xi|}}$$
and by Sobolev
$$ \|v\|_{L^{\frac{2n}{n-2+2\epsilon}}_x L^{\infty}_t} \lesssim \|v\|_{X^{1-\epsilon,\frac{1}{2}+}_{\tau =0}} \, . $$
In the case $n=3$ we estimate by Sobolev and (\ref{Tao})
$$ \big| \int\int uvw dx dt \big| \le \|u\|_{L^2_xL^2_t} \|v\|_{L^4_x L^{\infty}_t} \|w\|_{L^4_xL^2_t}
\lesssim \|u\|_{X^{0,0}_{\tau =0}} \|v\|_{X^{1-\epsilon,\frac{1}{2}+}_{|\tau|=|\xi|}} \|w\|_{X^{\frac{1}{4},\frac{1}{2}+}_{|\tau|=|\xi|}} $$
In order to obtain (\ref{52}) we estimate as follows:
$$ \big| \int\int uvw dx dt \big| \le \|u\|_{L^2_x L^2_t} \|v\|_{L^{\tilde{q}}_x L^{\infty}_t} \|w\|_{L^p_x L^2_t}$$
with $\frac{1}{\tilde{q}} = \frac{2(\frac{1}{2}-\alpha-\epsilon)}{n-1}$ and $ \frac{1}{p}= \frac{n-1-4(\frac{1}{2}-\alpha-\epsilon)}{2(n-1)}$. Then we use the embedding $H^{s+\alpha-\epsilon}_x \subset L^{\tilde{q}}_x$. This is true, because one easily checks $\frac{2(\frac{1}{2}-\alpha - \epsilon)}{n-1} \ge \frac{1}{2} - \frac{s+\alpha-\epsilon}{n}$, using $\alpha \le \frac{1}{4}$ and $s> \frac{n}{2}-\frac{3}{4}$. We next show that
$$\|w\|_{L^p_x L^2_t} \lesssim \|w\|_{X^{\frac{1}{2}-\alpha-\epsilon,\frac{1}{2}+}_{|\tau|=|\xi|}} \, . $$
This follows by interpolation between (\ref{Tao}) (with $k=0$) and the trivial identity $\|w\|_{L^2_x L^2_t} = \|u\|_{X^{0,0}_{|\tau|=|\xi|}} $ with interpolation parameter $\theta$ given by $\theta \frac{n-1}{2(n+1)} = \frac{1}{2}-\alpha-\epsilon$. One checks that $\theta < 1$ and $\frac{1}{p} = \frac{1-\theta}{2} + \theta \frac{n-1}{2(n+1)}$, so that (\ref{52}) follows.

Concerning the last term on the right hand side of (\ref{N4'}) we use
$\langle \xi_1 \rangle^{s-1+\alpha} \lesssim \langle \xi_2 \rangle^{s-1+\alpha} + \langle \xi_3 \rangle^{s-1+\alpha}$ so that we reduce to
\begin{equation}
\label{53}
\big|\int\int uvw dx dt\big| 
\lesssim \|u\|_{X^{0,\frac{1}{2}-\epsilon}_{\tau=0}} 
\|v\|_{X^{1-\epsilon,0}_{\tau=0}}
\|w\|_{X^{s-\frac{1}{2}+\epsilon,\frac{3}{4}+\epsilon}_{|\tau|=|\xi|}} 
\end{equation}
and
\begin{equation}
\label{54}
\big|\int\int uvw dx dt\big| 
\lesssim \|u\|_{X^{0,\frac{1}{2}-\epsilon}_{\tau=0}} 
\|v\|_{X^{s+\alpha-\epsilon,0}_{\tau=0}} \|w\|_{X^{\frac{1}{2}-\alpha+\epsilon,\frac{3}{4}+\epsilon}_{|\tau|=|\xi|}}  \,.
\end{equation}
In order to obtain (\ref{53}) in the case $n \ge 4$ we estimate by H\"older's inequality
$$ \big| \int\int uvw dx dt \big| \le \|u\|_{L^2_x L^{\frac{1}{\epsilon}}_t} \|v\|_{L^{\frac{2n}{n-2+2\epsilon}}_x L^2_t} \|w\|_{L^{\frac{n}{1-\epsilon}} L^{\frac{2}{1-2\epsilon}}_t} \, . $$
By Sobolev we have
$$ \|v\|_{L^{\frac{2n}{n-2+2\epsilon}}_x L^2_t} \lesssim \|v\|_{X^{1-\epsilon,0}_{\tau=0}} \,,$$
and by (\ref{Tao}) we obtain for $\frac{1}{p} = \frac{1}{n}-O(\epsilon)$ :
$$\|w\|_{L^p_x L^2_t} \lesssim \|w\|_{X^{k+\frac{n-1}{2(n+1)},\frac{1}{2}+}_{|\tau|=|\xi|}} \, , $$
where 
$$\frac{k}{n}=\frac{n-1}{2(n+1)}-\frac{1}{n}+ O(\epsilon) \, \Leftrightarrow \, k+\frac{n-1}{2(n+1)} = \frac{n}{2}- \frac{3}{2} + O(\epsilon) < s-\frac{3}{4} \, . $$
Interpolation with the standard Strichartz inequality (\ref{15}) for $q=r= \frac{2(n+1)}{n-1}$:
$$ \|w\|_{L_x^{\frac{2(n+1)}{n-1}} L_t^{\frac{2(n+1)}{n-1}}} = \|w\|_{L_t^{\frac{2(n+1)}{n-1}} L_x^{\frac{2(n+1)}{n-1}}} \lesssim \|w\|_{X^{\frac{1}{2},\frac{1}{2}+}_{|\tau|=|\xi|}} $$
and interpolation parameter $\theta = (n+1)\epsilon$ gives
$$\|w\|_{L^{\frac{n}{1-\epsilon}}_x L^{\frac{2}{1-2\epsilon}}_t} \lesssim \|w\|_{X^{s-\frac{3}{4},\frac{1}{2}+}_{|\tau|=|\xi|}} \, , $$
which is more than we need.\\
In order to obtain (\ref{53}) in the case $n = 3$ we estimate as follows:
\begin{align*}
 \big| \int\int uvw dx dt \big| &\lesssim \|u\|_{L^2_x L^{\frac{1}{\epsilon}}_t} \|v\|_{L^{4}_x L^2_t} \|w\|_{L^4 L^{\frac{2}{1-2\epsilon}}_t} \\
&\lesssim \|u\|_{X^{0,\frac{1}{2}-\epsilon}_{\tau=0}} 
\|v\|_{X^{1-\epsilon,0}_{\tau=0}}
\|w\|_{X^{\frac{1}{4}+\epsilon,\frac{1}{2}+\epsilon}_{|\tau|=|\xi|}} \, ,
\end{align*}
which is sufficient under our assumption $ s > \frac{3}{4} $ . \\
In order to obtain (\ref{54}) we estimate
$$ \big| \int\int uvw dx dt \big| \le \|u\|_{L^2_x L^{\frac{1}{\epsilon}}_t} \|v\|_{L^{p}_x L^2_t} \|w\|_{L^{q}_x L^{\frac{2}{1-2\epsilon}}_t} \, , $$
where $\frac{1}{p}= \frac{1}{2}-\frac{s+\alpha-\epsilon}{n}$ and $\frac{1}{q} = \frac{s+\alpha-\epsilon}{n}$, so that by Sobolev
$$\|v\|_{L^p_x L^2_t} \lesssim \|v\|_{X^{s+\alpha-\epsilon,0}_{\tau =0}} \, . $$
One easily checks that $\frac{1}{\tilde{q}} := \frac{1}{q}- O(\epsilon) > \frac{n-1}{2(n+1)}$ under our assumptions on $s$ and $\alpha$. By (\ref{Tao}) we obtain $$\|w\|_{L^{\frac{2(n+1)}{n-1}}_x L^2_t} \lesssim \|w\|_{X^{\frac{n-1}{2(n+1)},\frac{1}{2}+}_{|\tau|=|\xi|}} \, , $$
which we interpolate with the trivial identity $\|w\|_{L^2_xL^2_t} = \|w\|_{X^{0,0}_{|\tau|=|\xi|}} $, 
where the interpolation parameter $\theta$ is chosen such that
$$ \frac{1}{\tilde{q}} = \theta \frac{n-1}{2(n+1)} + (1-\theta) \frac{1}{2} \, \Leftrightarrow \, \theta = (n+1)(\frac{1}{2}-\frac{s+\alpha}{n}) + O(\epsilon) \, , $$
we obtain
$$\|w\|_{L^{\tilde{q}}_x L^2_t} \lesssim \|w\|_{X^{k,\frac{1}{2}+}_{|\tau|=|\xi|}} $$
with $k=\theta \frac{n-1}{2(n+1)} = \frac{n-1}{2}(\frac{1}{2}-\frac{s+\alpha}{n})+O(\epsilon) $. An easy calculation now shows that $k < \frac{1}{2}-\alpha$, so that another interpolation with Strichartz' inequality
$$ \|w\|_{L_x^{\frac{2(n+1)}{n-1}} L_t^{\frac{2(n+1)}{n-1}}} \lesssim \|w\|_{X^{\frac{1}{2},\frac{1}{2}+}_{|\tau|=|\xi|}} $$
and interpolation parameter $\theta = (n+1)\epsilon$ gives
$$\|w\|_{L^q_x L^{\frac{2}{1-2\epsilon}}_t} \lesssim \|w\|_{X^{k+O(\epsilon),\frac{1}{2}+}_{|\tau|=|\xi|}} \lesssim \|w\|_{X^{\frac{1}{2}-\alpha+\epsilon,\frac{1}{2}+}_{|\tau|=|\xi|}} \, .$$
This completes the proof of (\ref{18}). \\
{\bf Proof of (\ref{16}).} 
If $\widehat{\phi}$ is supported in $||\tau|-|\xi|| \gtrsim |\xi|$ we obtain
$$\|\phi\|_{X^{s+\alpha,\frac{1}{2}+\epsilon}_{\tau =0}}
\lesssim \|\phi\|_{X^{s,\frac{3}{4}+\epsilon}_{|\tau|=|\xi|}} \, . $$
which implies that (\ref{16}) follows from (\ref{18}), if $\widehat{\phi}_1$ or $\widehat{\phi}_2$ have this support property. So we may assume that both functions are supported in $||\tau|-|\xi|| \ll |\xi|$. This means that it suffices to show
$$ \int_* m(\xi_1,\xi_2,\xi_3,\tau_1,\tau_2,\tau_3) \prod_{i=1}^3 \widehat{u}_i(\xi_i,\tau_i) d\xi d\tau \lesssim \prod_{i=1}^3 \|u_i\|_{L^2_{xt}} \, , $$
where
$$m= \frac{|\tau_3|\chi_{||\tau_2|-|\xi_2|| \ll |\xi_2|} \chi_{||\tau_3|-|\xi_3|| \ll |\xi_3|}}{\langle \xi_1 \rangle^{1-\alpha-s} \langle \tau_1 \rangle^{\frac{1}{2}-\epsilon-} \langle \xi_2 \rangle^s \langle |\tau_2|-|\xi_2| \rangle^{\frac{3}{4}+\epsilon} \langle \xi_3 \rangle^s \langle |\tau_3|-|\xi_3|\rangle^{\frac{3}{4}+\epsilon}} \, . $$
Since $\langle \tau_3 \rangle \sim \langle \xi_3 \rangle$ , $\langle \tau_2 \rangle \sim \langle \xi_2 \rangle$ and $\tau_1+\tau_2+\tau_3=0$ we have 
\begin{equation}
|\tau_3| \lesssim \langle \tau_1 \rangle^{\frac{1}{2}-\epsilon-} \langle \xi_3 \rangle^{\frac{1}{2}+\epsilon+} +\langle \xi_2 \rangle^{\frac{1}{2}-\epsilon-} \langle \xi_3 \rangle^{\frac{1}{2}+\epsilon+} , 
\end{equation}
Concerning the first term on the right hand side we have to show 
$$\big|\int \int uvw dx dt\big| \lesssim \|u\|_{X^{1-\alpha-s,0}_{\tau=0}} \|v\|_{X^{s,\frac{3}{4}+\epsilon}_{|\tau|=|\xi|}} \|w\|_{X^{s-\frac{1}{2}-\epsilon-,\frac{3}{4}+\epsilon}_{|\tau|=|\xi|}} \, .$$
We use  \cite{FK} ,Thm. 1.1  , which shows
$$\|vw\|_{L^2_t H^{s-\frac{3}{4}}_x} \lesssim \|v\|_{X^{s,\frac{1}{2}+}_{|\tau|=|\xi|}} \|w\|_{X^{s-\frac{1}{2}-\epsilon-,\frac{1}{2}+}_{|\tau|=|\xi|}} $$
under the assumption $s > \frac{n}{2}-\frac{3}{4}$. This is enough, because $\alpha \le \frac{1}{4}$.\\
Concerning the second term on the right hand side we use $\langle \xi_1 \rangle^{s-1+\alpha} \lesssim \langle \xi_2 \rangle^{s-1+\alpha} + \langle \xi_3 \rangle^{s-1+\alpha}$ , so that we reduce to
$$
\big|\int\int uvw dx dt\big| 
\lesssim \|u\|_{X^{0,\frac{1}{2}-\epsilon-}_{\tau=0}} 
\|v\|_{X^{\frac{1}{2}-\alpha+\epsilon+,\frac{3}{4}+\epsilon}_{|\tau|=|\xi|}}
\|w\|_{X^{s-\frac{1}{2}-\epsilon-,\frac{3}{4}+\epsilon}_{|\tau|=|\xi|}} 
$$
and
$$
\big|\int\int uvw dx dt\big| 
\lesssim \|u\|_{X^{0,\frac{1}{2}-\epsilon}_{\tau=0}}
\|v\|_{X^{s-\frac{1}{2}+\epsilon+,\frac{3}{4}+\epsilon}_{|\tau|=|\xi|}}  \|w\|_{X^{\frac{1}{2}-\alpha-\epsilon-,\frac{3}{4}+\epsilon}_{|\tau|=|\xi|}} \, .
$$
We even show the slightly stronger estimate
$$
\big|\int\int uvw dx dt\big| 
\lesssim \|u\|_{X^{0,\frac{1}{2}-\epsilon-}_{\tau=0}} 
\|v\|_{X^{\frac{1}{2}-\alpha-\epsilon-,\frac{3}{4}+\epsilon}_{|\tau|=|\xi|}}
\|w\|_{X^{s-\frac{1}{2}-\epsilon-,\frac{3}{4}+\epsilon}_{|\tau|=|\xi|}} \, ,
$$
which implies both. We start with the estimate
\begin{align*}
\big|\int \int uvw dx dt\big|
 \lesssim \|u\|_{L^2_x L^{\frac{1}{\epsilon}-}_t} \|v\|_{L^p_x L^{\frac{2}{1-2\epsilon}+}_t} \|w\|_{L^q_x L^2_t} \, ,
\end{align*}
where $\frac{1}{p} + \frac{1}{q} = \frac{1}{2}$.  Interpolating (\ref{Tao})
$$ \|v\|_{L^{\frac{2(n+1)}{n-1}}_x L^2_t} \lesssim \|v\|_{X^{\frac{n-1}{2(n+1)},\frac{1}{2}+}_{|\tau|=|\xi|}} $$
with the trivial identity $\|v\|_{L^2_x L^2_t} = \|v\|_{X^{0,0}_{|\tau|=|\xi|}}$ with interpolation parameter $\theta$ given by $\theta \frac{n-1}{2(n+1)} = \frac{1}{2}-\alpha-2\epsilon$ (where we remark that $\theta < 1$) this gives
$$\|v\|_{L^{\tilde{p}}_x L^2_t}  \lesssim \|v\|_{X^{\frac{1}{2}-\alpha-O(\epsilon),\frac{1}{2}+}_{|\tau|=|\xi|}} \, , $$
where
 $$\frac{1}{\tilde{p}} = \frac{n-1}{2(n+1)} \theta + (1-\theta)\frac{1}{2} = \frac{n-3}{2(n-1)} + \frac{2}{n-1}\alpha + O(\epsilon) \, . $$
 Interpolating this estimate with Strichartz' estimate just slightly changing the parameters we obtain
 $$ \|v\|_{L^p_x L^{\frac{2}{1-2\epsilon}+}_x} \lesssim \|v\|_{X^{\frac{1}{2}-\alpha-\epsilon-,\frac{1}{2}+}_{|\tau|=|\xi|}} \, , $$
 where $\frac{1}{p}= \frac{1}{\tilde{p}}+ O(\epsilon)$. Thus $\frac{1}{q} = \frac{1}{2}-\frac{1}{p}= \frac{1}{n-1} - \frac{2}{n-1}\alpha - O(\epsilon)$.\\
Next we apply (\ref{Tao}) to obtain
$$ \|w\|_{L^q_x L^2_t} \lesssim \|w\|_{X^{k+\frac{n-1}{2(n+1)},\frac{1}{2}+}_{|\tau|=|\xi|}} $$
with
$$ \frac{1}{q} = \frac{n-1}{2(n+1)} - \frac{k}{n} \, \Leftrightarrow \, k = n(\frac{n-1}{2(n+1)} - \frac{1}{n-1} + \frac{2}{n-1}\alpha) + O(\epsilon) \, . $$
In order to conclude the desired estimate
$$ \|w\|_{L^q_x L^2_t} \lesssim \|w\|_{X^{s-\frac{1}{2}-\epsilon-,\frac{1}{2}+}_{|\tau|=|\xi|}} $$
we need
\begin{equation}
\label{*****}
 s \ge k + \frac{n}{n+1} + O(\epsilon) = \frac{n}{2} - \frac{n}{n-1} + \frac{2n}{n-1}\alpha + O(\epsilon) \, . 
 \end{equation}
This means that in order to obtain a minimal lower bound for $s$ one should also minimize $\alpha$. On the other hand in the proof of (\ref{30}) below we have to maximize $\alpha$.
Comparing condition (\ref{*****}) with (\ref{****}) below we optimize $\alpha$ by choosing
\begin{equation}
\label{*}
\frac{n}{2} - \frac{n}{n-1} + \frac{2n}{n-1}\alpha = \frac{n}{2} - \frac{1}{4} -2\alpha \, \Leftrightarrow \, \alpha = \frac{3n+1}{8(2n-1)} \, , 
\end{equation}
which leads to our choice of $\alpha$. Thus the condition on $s$ reduces to
$$ s \ge  \frac{n}{2} - \frac{1}{4} -2\alpha +O(\epsilon) = \frac{n}{2} - \frac{5}{8} - \frac{5}{8(2n-1)} + O(\epsilon) \, . $$
This is exactly our assumption on $s$.\\
{\bf Proof of (\ref{19}):} Sobolev's multiplication law shows the estimate
$$ \| |\nabla|^{-1} (A_1 \partial_t A_2)\|_{C^0(H^{s-1})} \lesssim \|A_1\|_{C^0(H^s)} \|\partial_t A_2\|_{C^0(H^{s-1})}$$
for $s > \frac{n}{2}-1$. Use now $$ A=A^{cf} + \sum_{\pm} A^{df}_{\pm} \quad , \quad \partial_t A = \partial_t A^{cf} + i \langle \nabla \rangle(A_+^{df} -A_-^{df}) \, $$
from which the estimate (\ref{19}) easily follows. \\
{\bf Proof of (\ref{29}):} This a generalization of the proof given by  Tao (\cite{T1}) in dimension $n=3$. We have to show
$$
\int_* m(\xi,\tau) \prod_{i=1}^3 \widehat{u}_i(\xi_i,\tau_i)  d\xi d\tau \lesssim \prod_{i=1}^3 \|u_i\|_{L^2_{xt}} \, , 
$$
where $\xi=(\xi_1,\xi_2,\xi_3) \, , \,\tau=(\tau_1,\tau_2,\tau_3)$ , * denotes integration over $ \sum_{i=1}^3 \xi_i = \sum_{i=1}^3 \tau_i = 0$ , and 
$$ m = \frac{(|\xi_2|+|\xi_3|) \langle \xi_1 \rangle^{s-1} \langle |\tau_1|-|\xi_1|) \rangle^{-\frac{1}{4}+2\epsilon}}{\langle \xi_2 \rangle^s \langle |\tau_2| - |\xi_2|\rangle^{\frac{3}{4}+\epsilon}  \langle \xi_3 \rangle^{s+\alpha}\langle \tau_3 \rangle^{\frac{1}{2}+\epsilon}} \, .$$
Case 1: $|\xi_2| \le |\xi_1|$ ($\Rightarrow$ $|\xi_2|+|\xi_3| \lesssim |\xi_1|$). \\
By two applications of the averaging principle (\cite{T}, Prop. 5.1) we may replace $m$ by
$$ m' = \frac{ \langle \xi_1 \rangle^s \chi_{||\tau_2|-|\xi_2||\sim 1} \chi_{|\tau_3| \sim 1}}{ \langle \xi_2 \rangle^s \langle \xi_3 \rangle^{s+\alpha}} \, . $$
Let now $\tau_2$ be restricted to the region $\tau_2 =T + O(1)$ for some integer $T$. Then $\tau_1$ is restricted to $\tau_1 = -T + O(1)$, because $\tau_1 + \tau_2 + \tau_3 =0$, and $\xi_2$ is restricted to $|\xi_2| = |T| + O(1)$. The $\tau_1$-regions are essentially disjoint for $T \in {\mathbb Z}$ and similarly the $\tau_2$-regions. Thus by Schur's test (\cite{T}, Lemma 3.11) we only have to show
\begin{align*}
 &\sup_{T \in {\mathbb Z}} \int_* \frac{\langle \xi_1 \rangle^s \chi_{\tau_1=-T+O(1)} \chi_{\tau_2=T+O(1)} \chi_{|\tau_3|\sim 1} \chi_{|\xi_2|=|T|+O(1)}}{\langle \xi_2 \rangle^s \langle \xi_3 \rangle^{s+\alpha}} \prod_{i=1} \widehat{u}_i(\xi_i,\tau_i)  d\xi d\tau  \\
 & \hspace{25em} \lesssim \prod_{i=1}^3 \|u_i\|_{L^2_{xt}} \, . 
\end{align*}
The $\tau$-behaviour of the integral is now trivial, thus we reduce to
\begin{equation}
\label{55}
\sup_{T \in {\mathbb N}} \int_{\sum_{i=1}^3 \xi_i =0}  \frac{ \langle \xi_1 \rangle^s \chi_{|\xi_2|=|T|+O(1)}}{ \langle T \rangle^s \langle \xi_3 \rangle^{s+\alpha}} \widehat{f}_1(\xi_1)\widehat{f}_2(\xi_2)\widehat{f}_3(\xi_3)d\xi \lesssim \prod_{i=1}^3 \|f_i\|_{L^2_x} \, .
\end{equation}
Assuming now $|\xi_3| \le |\xi_1|$ (the other case being simpler) 
it only remains to consider the following two cases: \\
Case 1.1: $|\xi_1| \sim |\xi_3| \gtrsim T$. We obtain in this case
\begin{align*}
L.H.S. \, of \, (\ref{55}) 
&\lesssim \sup_{T \in{\mathbb N}} \frac{1}{T^{s+\alpha}} \|f_1\|_{L^2} \|f_3\|_{L^2} \| {\mathcal F}^{-1}(\chi_{|\xi|=T+O(1)} \widehat{f}_2)\|_{L^{\infty}({\mathbb R}^n)} \\
&\lesssim \sup_{T \in{\mathbb N}} \frac{1}{ T^{s+\alpha}} 
\|f_1\|_{L^2} \|f_3\|_{L^2} \| \chi_{|\xi|=T+O(1)} \widehat{f}_2\|_{L^1({\mathbb R}^n)} \\
&\lesssim \hspace{-0.1em}\sup_{T \in {\mathbb N}} \frac{T^{\frac{n-1}{2}}}{T^{s+\alpha}}  \prod_{i=1}^3 \|f_i\|_{L^2} \lesssim\hspace{-0.1em}
\prod_{i=1}^3 \|f_i\|_{L^2} \, ,
\end{align*}
because one easily calculates that $2(s+\alpha) > n-1$ under our choice of $s$ and $\alpha$.
Case 1.2: $|\xi_1| \sim T \gtrsim |\xi_3|$. 
An elementary calculation shows that
\begin{align*}
L.H.S. \, of \,  (\ref{55})
\lesssim \sup_{T \in{\mathbb N}} \| \chi_{|\xi|=T+O(1)} \ast \langle \xi \rangle^{-2(s+\alpha)}\|^{\frac{1}{2}}_{L^{\infty}(\mathbb{R}^{n-1})} \prod_{i=1}^3 \|f_i\|_{L^2_x} \lesssim \prod_{i=1}^3 \|f_i\|_{L^2_x} \, ,
\end{align*}
using as in case 1.1 that $2(s+\alpha) > n-1$ ,
so that the desired estimate follows.\\
Case 2. $|\xi_1| \le |\xi_2|$ ($\Rightarrow$ $|\xi_2|+|\xi_3| \lesssim |\xi_2|$). \\
Exactly as in case 1 we reduce to
$$
\sup_{T \in {\mathbb N}} \int_{\sum_{i=1}^3 \xi_i =0}  \frac{ \langle \xi_1 \rangle^{s-1} \chi_{|\xi_2|=|T|+O(1)}}{ \langle T \rangle^{s-1} \langle \xi_3 \rangle^{s+\alpha}} \widehat{f}_1(\xi_1)\widehat{f}_2(\xi_2)\widehat{f}_2(\xi_3)d\xi \lesssim \prod_{i=1}^3 \|f_i\|_{L^2_x} \, .
$$
This can be treated as in case 1.\\
{\bf Proof of (\ref{30}):} By Sobolev's multiplication law we obtain
$$
 |\int \int fgh dx dt| \lesssim \|f\|_{X^{s+\alpha,\frac{1}{2}+\epsilon}_{\tau=0}}
\|g\|_{X^{s+\alpha -1,\frac{1}{2}+\epsilon}_{\tau=0}}
\|h\|_{X^{-s+\frac{5}{4}-2\epsilon,-\frac{1}{2}}_{\tau=0}} \, ,
$$
where we need that 
\begin{equation}
\label{****}
s+2\alpha+\frac{1}{4}-2\epsilon > \frac{n}{2} \, ,
\end{equation}
which holds under our assumptions on $s$ and $\alpha$.
Using the elementary estimate 
$$ \frac{\langle \xi \rangle^{\frac{1}{4}-2\epsilon}}{\langle \tau \rangle^{\frac{1}{4}-2\epsilon}} \lesssim \langle |\tau|-|\xi| \rangle^{\frac{1}{4}-2\epsilon}$$
 we obtain 
$$\|h\|_{X^{-s+\frac{5}{4}-2\epsilon,-\frac{1}{2}}_{\tau=0}}  \lesssim \|h\|_{[X^{1-s,\frac{1}{4}-2\epsilon}_{|\tau|=|\xi|}} \,$$
which implies (\ref{30}).
\\
{\bf Proof of (\ref{31}):} We use the following consequences of Sobolev's embedding and Strichartz' inequality:
\begin{align}
\|A\|_{L^{\infty}_t H^{s+\alpha}_x} & \lesssim \|A\|_{X^{s+\alpha,\frac{1}{2}+}_{\tau=0}} \, , \\
\label{56}
\|A\|_{L^{4-}_t H^{1-s}_x} & \lesssim \|A\|_{X^{1-s,\frac{1}{4}-}_{|\tau|=|\xi|}} \\
\label{Str}
\|A\|_{L^4_t H^{s-\frac{n+1}{4(n-1)},\frac{2(n-1)}{n-2}}_x} & \lesssim \|A\|_{X^{s,\frac{1}{2}+}_{|\tau|=|\xi|}} \, ,
\end{align}
where we applied (\ref{15}) with $q=4$ , $r= \frac{2(n-1)}{n-2}$ , $\mu = \frac{n+1}{4(n-1)}$ and also
\begin{equation}
\label{57}
\|A\|_{L^{4+}_t H^{s-\frac{n+1}{4(n-1)}+,\frac{2(n-1)}{n-2}-}_x}  \lesssim \|A\|_{X^{s,\frac{1}{2}+}_{|\tau|=|\xi|}} \, . 
\end{equation}
Assume now $s \ge 1$.
Taking the dual of (\ref{56}) we obtain
$$
\|A_1 A_2 A_3\|_{X^{s-1,-\frac{1}{4}+}_{|\tau|=|\xi|}} \lesssim \|A_1 A_2 A_3\|_{L^{\frac{4}{3}+}_t H^{s-1}_x} \, .
$$
This can be estimated by
$$
 \|A_1\|_{L^{4+}_t H^{s-1,p-}_x} \|A_2\|_{L^4_t L^{\frac{2n}{1+\alpha}+}_x} \|A_3\|_{L^4_t L^{\frac{2n}{1+\alpha}+}_x} 
$$ 
where $\frac{1}{p} = \frac{1}{2} - \frac{1+\alpha}{n}$ , and similar terms with reversed roles of $A_j$ 
 . Now by Sobolev we have
$ H^{s+\alpha,2}_x \subset H^{s-1,p-}_x$ , so that
$$\|A_1\|_{L^{4+}_t H^{s-1,p-}_x} \lesssim \|A_1\|_{X^{s+\alpha,\frac{1}{2}+}_{\tau=0}} \, . $$
Next we obtain $H^{s-\frac{n+1}{4(n-1)}+,\frac{2(n-1)}{n-2}-}_x \subset H^{s-1,p-}_x$ , because
$$\frac{1}{p} > \frac{n-2}{2(n-1)} - \frac{1}{n}(1-\frac{n+1}{4(n-1)}) \, \Leftrightarrow \frac{3n+1}{4n(n-1)} > \frac{\alpha}{n} \, , $$
which holds, because $\frac{1}{4} \ge \alpha = \frac{3n+1}{8(2n-1)} \ge \frac{3}{16}$ . This implies by (\ref{57}) 
$$\|A_1\|_{L^{4+}_t H^{s-1,p-}_x} \lesssim \|A_1\|_{X^{s,\frac{1}{2}+}_{|\tau|=|\xi|}} \, . $$
Next we have $H^{s+\alpha,2}_x \subset L^{\frac{2n}{1+\alpha}+}_x $ , because the inequality 
$$\frac{1+\alpha}{2n} > \frac{1}{2} -\frac{s+\alpha}{n}$$
holds by $s> \frac{n}{2} - \frac{3}{4}$ and $\alpha \ge \frac{3}{16}$ . This implies
$$\|A_j\|_{L^4_t L^{\frac{2n}{1+\alpha}+}_x} \lesssim \|A_j\|_{L^4_t H^{s+\alpha,2}_x }\lesssim \|A_j\|_{X^{s+\alpha,\frac{1}{2}+}_{\tau=0}} \, . $$
Finally by Sobolev $H^{s-\frac{n+1}{4(n-1)},\frac{2(n-1)}{n-2}}_x \subset L^{\frac{2n}{1+\alpha}+}_x$ , because one easily calculates that
$ \frac{1+\alpha}{2n} > \frac{n-2}{2(n-1)} - \frac{1}{n}(s-\frac{n+1}{4(n-1)}) $ using $s \ge \frac{n}{2}-\frac{3}{4}$ and $\alpha \ge \frac{3}{16}$. Thus 
by (\ref{Str})
$$\|A_j\|_{L^4_t L^{\frac{2n}{1+\alpha}+}_x} \lesssim \|A_j\|_{L^4_t H^{s-\frac{n+1}{4(n-1)},\frac{2(n-1)}{n-2}}_x} \lesssim \|A_j\|_{X^{s,\frac{1}{2}+}_{|\tau| = |\xi|}} \, . $$
This completes the proof of (\ref{31}) for $s \ge 1$. It remains to consider the case $1>s> \frac{3}{4}$ in dimension $n=3$ and $\alpha = \frac{1}{4}$. This case is much easier. We only use
$$ \|A\|_{L^{4-}_t L^2_x} \lesssim \|A\|_{X^{0,\frac{1}{4}-}_{|\tau|=|\xi|}} \lesssim \|A\|_{X^{1-s,\frac{1}{4}-}_{|\tau|=|\xi|}} \, , $$
so that by duality
$$ \|A_1 A_2 A_3\|_{X^{s-1,-\frac{1}{4}+}_{|\tau|=|\xi|}} \lesssim  \|A_1 A_2 A_3\|_{L^{\frac{4}{3}+}_t L^2_x} \lesssim \prod_{I=1}^3 \|A_i\|_{L^{4+}_t L^6_x} \, . $$
Now by Sobolev for $s>\frac{3}{4}$ we obtain
$$ \|A_i\|_{L^{4+}_t L^6_x} \lesssim \|A_i\|_{L^{4+}_t H^1_x} \lesssim \|A_i\|_{X^{s+\frac{1}{4},\frac{1}{2}+}_{\tau=0}} \, , $$
and using Sobolev's embedding and
Strichartz' inequality (\ref{15}) gives
$$ \|A_i\|_{L^{4+}_t L^6_x} \lesssim \|A_i\|_{L^{4+}_t H^{\frac{1}{4}+,4-}_x} \lesssim \|A_i\|_{X^{\frac{3}{4}+,\frac{1}{2}+}_{|\tau|=|\xi|}} \lesssim  \|A_i\|_{X^{s,\frac{1}{2}+}_{|\tau|=|\xi|}} \, . $$ \\
{\bf Proof of (\ref{32}):} The case $N = 3$ reduces to (\ref{31}). Next we consider  the case $N=4$ in dimension $n=3$. We may assume $s\le 1$, because the general case can be reduced to this case easily. This follows from Prop. \ref{Prop.2} as follows:
$$ \| |\phi|^3 \phi\|_{X^{s-1,-\frac{1}{4}+}_{|\tau|=|\xi|}} \lesssim \| |\phi|^3 \phi\|_{L^{\frac{4}{3}+}_t H^{s-1}_x} \lesssim \| |\phi|^3 \phi\|_{L^{\frac{4}{3}+}_t L^p_x} \lesssim \|\phi\|_{L^{\frac{16}{3}+}_t L^{4p}_x}^4 \, ,$$
where $\frac{1}{p} = \frac{1}{2}-\frac{s-1}{3}$ . We now use Strichartz estimate (\ref{15})  with $q=\frac{16}{3}+$, $r=\frac{16}{5}-$,  $\mu=\frac{3}{8}+$ to conclude
$$\| |\phi|^3 \phi\|_{X^{s-1,-\frac{1}{4}+}_{|\tau|=|\xi|}} \lesssim   \|\phi\|_{L^{\frac{16}{3}+}_t H^{l,\frac{16}{5}-}_x}^4 \lesssim  \|\phi\|_{X^{l+\mu,\frac{1}{2}+}_{|\tau|=|\xi|}}^4 \lesssim  \|\phi\|_{X^{s,\frac{1}{2}+}_{|\tau|=|\xi|}}^4 \, , $$
provided $H^{l,\frac{16}{5}-}_x \subset L^{4p}_x$, which is fulfilled, if $l=\frac{5+4s}{16}+$, so that $l+\mu \le s$, if $\frac{5+4s}{16} + \frac{3}{8} < s \, \Leftrightarrow \, s > \frac{11}{12}$, which is equivalent to our assumption $N < 1+\frac{7}{4(\frac{n}{2}-s)}$. The case $N=2$ for $n=3$ is much easier handled by the standard Strichartz inequality:
$$ \| |\phi| \phi \|_{H^{s-1,-\frac{1}{4}+}_{|\tau|=|\xi|}} \lesssim \|\phi\|_{L^4_t L^4_x}^2 \lesssim \|\phi\|_{X^{\frac{1}{2},\frac{1}{2}+}_{|\tau|=|\xi|}} \, . $$
In all the other cases under our assumptions we have $s\ge 1$. We have
$$ \|  |\phi|^{N-1} \phi\|_{X^{s-1,-\frac{1}{4}+}_{|\tau|=|\xi|}} \lesssim \| |\phi|^{N-1} \phi\|_{L^{\frac{4}{3}+}_t H^{s-1}_x} \lesssim  \|\phi\|_{L^{\frac{4}{3}N+}_t L^{\tilde{q}}_x}^{N-1} \|\phi\|_{L^{\frac{4}{3}N+}_t H^{s-1,p}_x} \, .$$
Here $\frac{1}{p} + \frac{N-1}{\tilde{q}} = \frac{1}{2}$. We obtain $H^{s-1,p} \subset L^{\tilde{q}}$ , if $\frac{1}{\tilde{q}} = \frac{1}{p} - \frac{s-1}{n}$ , so that
$$ \frac{1}{p} = \frac{1}{N}\Big(\frac{1}{2}+\frac{(N-1)(s-1)}{n}\Big) \, , $$
thus
$$  \| |\phi|^{N-1} \phi\|_{X^{s-1,-\frac{1}{4}+}_{|\tau|=|\xi|}} \lesssim   \|\phi\|_{L^{\frac{4}{3}N+}_t H^{s-1,p}_x}^N \, . $$
The case $N=2$ is again easy. In this case we have $ \frac{1}{p} = \frac{s-1}{2n} + \frac{1}{4}$, which implies by Sobolev $H^{s,2} \subset H^{s-1,p}$ under the condition $\frac{1}{p} \ge \frac{1}{2}-\frac{1}{n}$, which is easily seen to be equivalent to $s \ge \frac{n}{2} -1$, which certainly holds, so that we obtain the desired bound $\|\phi\|_{X^{s,\frac{1}{2}+}_{|\tau|=|\xi|}}^2$.

It remains to consider $N\ge 4$. We use Strichartz' estimate (\ref{15}) with $q=\frac{4}{3}N+$,  $\frac{1}{r} = \frac{1}{2} - \frac{3}{2N(n-1)} +$ , $\mu = n(\frac{1}{2}-\frac{1}{r}) - \frac{1}{q} = \frac{3(n+1)}{4N(n-1)} +$ to conclude
\begin{equation}
\label{60}
 \| |\phi|^{N-1} \phi\|_{X^{s-1,-\frac{1}{4}+}_{|\tau|=|\xi|}} \lesssim  \|\phi\|_{L^{\frac{4}{3}N+}_t H^{l,r}_x}^N \lesssim   \|\phi\|_{X^{l+\mu,\frac{1}{2}+}_{|\tau|=|\xi|}}^N \lesssim    \|\phi\|_{X^{s,\frac{1}{2}+}_{|\tau|=|\xi|}}^N  \, ,
\end{equation}
if we $H^{l,r} \subset H^{s-1,p}$ and  $l+\mu \le s$. By Sobolev we need
$$ \frac{1}{r} \ge \frac{1}{p} \ge \frac{1}{r}-\frac{l-s+1}{n} \, . $$
We calculate
\begin{align}
\nonumber
\frac{1}{r} \ge \frac{1}{p} \, &\Leftrightarrow \, \frac{1}{2} - \frac{3}{2N(n-1)} > \frac{1}{N}\Big(\frac{1}{2}+\frac{(N-1)(s-1)}{n}\Big) \\
\label{61}
 &\Leftrightarrow \, s < \frac{n}{2} + 1 - \frac{3n}{(N-1)2(n-1)} \, .
\end{align}
In this case we can choose $l=\frac{n}{r}-\frac{n}{p}+s-1$ , so that one easily calculates
\begin{align*}
&l+\mu \le s \\
&\, \Leftrightarrow n\Big(\frac{1}{2}-\frac{3}{2N(n-1)}\Big) - \frac{n}{N}\Big(\frac{1}{2}+\frac{(N-1)(s-1)}{n}\Big)+s-1 + \frac{3(n+1)}{4N(n-1)} < s \\ &\Leftrightarrow \, s > \frac{n}{2}-\frac{7}{4(N-1)} \, \Leftrightarrow \, \Big( N < 1+\frac{7}{4(\frac{n}{2}-s)}  \,\, {\mbox if}\, s < \frac{n}{2} \,\, {\mbox and} \, N < \infty \,\, {\mbox if} \, s \ge \frac{n}{2} \Big) \, . 
\end{align*}
This is exactly our assumption on $s$ and $N$. This lower bound on $s$ and also the lower bound on $s$ in Prop \ref{Prop} is compatible with the upper bound (\ref{61}) in our case $N \ge 4$ and $n \ge 3$, as an easy calculation shows. As always the desired estimate (\ref{60}) for greater $s$ can be reduced to this case so that (\ref{61}) is redundant.
Thus (\ref{60}) is proven. This completes the proof of (\ref{32}) and also the proof of Prop. \ref{Prop} and Prop. \ref{Prop'}.
\end{proof}

\section{Removal of the assumption $A^{cf}(0)=0$}
Applying an idea of Keel and Tao \cite{T1} we use the gauge invariance of the Yang-Mills-Higgs system to show that the condition $A^{cf}(0)=0$, which had to be assumed in Prop. \ref{Prop}, can be removed. A completely analogous result holds for the Yang-Mills equation and Prop. \ref{Prop'}.
\begin{lemma}
\label{Lemma} Let $ n\ge 3$ , $s>\frac{n}{2}-\frac{3}{4}$ and $0 < \epsilon \ll 1$. Assume 
 $(A,\phi)\in( C^0([0,1],H^s) \cap C^1([0,1],H^{s-1}) \times (C^0([0,1],H^s) \cap C^1([0,1],H^{s-1}))$ , $A_0 = 0$ and
\begin{equation}
\label{***}
\|A^{df}(0)\|_{H^s} + \|(\partial_t A)^{df}(0)\|_{H^{s-1}} + \|A^{cf}(0)\|_{H^s} + \|\phi(0)\|_{H^s} + \|(\partial_t \phi)(0)\|_{H^{s-1}} \le \epsilon \, .
\end{equation}
Then there exists a gauge transformation $T$ preserving the temporal gauge such that $(TA)^{cf}(0) = 0$ and
\begin{align}
\label{T1}
\|(TA)^{df}(0)\|_{H^s} + \|(\partial_t TA)^{df}(0)\|_{H^{s-1}} + \|(T \phi)(0)\|_{H^s} + \|(\partial_t T\phi)(0)\|_{H^{s-1}} \lesssim \epsilon \, .
\end{align}
T preserves also the regularity, i.e. $TA\in C^0([0,1],H^s) \cap C^1([0,1],H^{s-1})$ , $T\phi \in C^0([0,1],H^s) \cap C^1([0,1],H^{s-1})$. If
$A \in X^{s,\frac{3}{4}+}_+[0,1] + X^{s,\frac{3}{4}+}_-[0,1] + X^{s+\alpha,\frac{1}{2}+}_{\tau=0}[0,1]$, where $\alpha = \frac{3n+1}{8(2n-1)}$ , $\partial_t A^{cf} \in C^0([0,1],H^{s-1})$ and $\phi\in X^{s,\frac{3}{4}+}_+[0,1] + X^{s,\frac{3}{4}+}_-[0,1]$ , then $TA$ , $T\phi$ belong to the same spaces. Its inverse $T^{-1}$ has the same properties.
\end{lemma}

In the proof we frequently use 
\begin{lemma} Let $ n \ge 3$ , $s > \frac{n}{2}-1$ 
and define $\|f\|_X := \|\nabla f\|_{H^s}$ .
The following estimates hold:
\begin{align*}
\| fg \|_X &\le c_1 \|f\|_X \|g\|_X \\
\| fg \|_{H^s} &\le c_1 \|f\|_X \|g\|_{H^s} \\
\| fg \|_{H^{s-1}} &\le c_1 \|f\|_X \|g\|_{H^{s-1}} \, .
\end{align*}
\end{lemma}
\begin{proof}
This follows essentially by Sobolev's multiplication law, where we remark that the singularity of $|\nabla|^{-1}$ is harmless in dimension $n \ge 3$.
\end{proof}

\begin{proof}[Proof of Lemma \ref{Lemma}]
This is achieved by an iteration argument. Assume that one has besides (\ref{***}):
\begin{equation}
\label{42}
\|A^{cf}(0)\|_{H^s} \le \delta
\end{equation}
for some $0<\delta\le \epsilon$.  In the first step we set $\delta=\epsilon$, so that the condition is fulfilled, in the next steps $\delta = \epsilon^{\frac{3}{2}}$ , $\delta=\epsilon^2$ etc. We use the Hodge decomposition of $A$:
$$A=A^{cf}+A^{df} = (-\Delta)^{-1} \nabla \,div \,A + A^{df}\, . $$
We define $V_1 := - (-\Delta)^{-1}div \,A(0)$ , so that $\nabla V_1 = A^{cf}(0)$. Thus $$ \|V_1\|_X := \| \nabla V_1\|_{H^s} = \|A^{cf}(0)\|_{H^s} \le \delta \, . $$ We define $U_1 := \exp(V_1)$ and consider the gauge transformation $T_1$ with
\begin{align*}
A_0 & \longmapsto U_1 A_0 U_1^{-1} - (\partial_t U_1) U_1^{-1} \\
A & \longmapsto U_1 A U_1^{-1} - (\nabla U_1) U_1^{-1} \\
\phi & \longmapsto U_1 \phi U_1^{-1} \, .
\end{align*}
Then $T_1$ preserves the temporal gauge, because $U_1$ is independent of $t$, a property, which is true for all the gauge transformations in the sequel as well. Moreover
\begin{align}
\nonumber
(T_1 A)(0) & = \exp V_1 A(0) \exp (-V_1) - \nabla(\exp V_1) \exp(-V_1) \\ \nonumber
& = A^{df}(0) + (\exp V_1 A^{df}(0) \exp(-V_1) - A^{df}(0)) \\ \label{40}
&  \hspace{1em} + (\exp V_1 A^{cf}(0) - \nabla(\exp V_1)) \exp(-V_1) 
\end{align}
and thus
\begin{align}
\nonumber
(T_1 A)^{cf}(0) 
& =-(-\Delta)^{-1} \nabla div(\exp V_1 A^{df}(0) \exp(-V_1) - A^{df}(0)) \\ \label{53'}
& \hspace{1em}-(-\Delta)^{-1} \nabla div ((\exp V_1 A^{cf}(0) - \nabla(\exp V_1)) \exp(-V_1)) \, . 
\end{align}
Using a Taylor expansion and Lemma \ref{Lemma} we obtain
\begin{align*}
 & \| \exp V_1 A^{df}(0) \exp(-V_1) - A^{df}(0)\|_{H^s} \\
 &\lesssim \|(\exp V_1-I) A^{df}(0) (\exp(-V_1)-I)\|_{H^s} 
+\|A^{df}(0) (\exp(-V_1)-I)\|_{H^s} \\
& \hspace{1em} + \|(\exp V_1-I)A^{df}(0) \|_{H^s} \\
 &  \lesssim (\|\exp V_1-I\|_X +1)\|A^{df}(0)\|_{H^s} \|\exp(-V_1)-I\|_X \\
& \hspace{1em}+ \|\exp V_1-I\|_X \|A^{df}(0)\|_{H^s} \\
 &  \lesssim(1+\delta) \epsilon \delta \\
& \le \frac{c_0}{2} \epsilon \delta \, .
\end{align*}
We used the estimate
\begin{align*}
\| \exp V_1 - I\|_X & \le \sum_{k=1}^{\infty} \frac{\|V_1^k\|_X}{k !} \le \sum_{k=1}^{\infty} \frac{(c_1 \|V_1\|_X)^k}{c_1 k !} = c_1^{-1}(\exp(c_1 \|V_1\|_X) - 1) \\
&\le c_1^{-1}(\exp(c_1 \delta) -1) \lesssim \delta \, . 
\end{align*}
Furthermore we obtain
\begin{align*}
&\|\exp V_1 A^{cf}(0) - \nabla(\exp V_1)) \|_{H^s} = \| \sum_{k=0}^{\infty} \frac{V_1^k}{k !} \nabla V_1 - \sum_{k=1}^{\infty} \frac{\nabla(V_1^k)}{k!} \|_{H^s} \\
& =\| \sum_{k=1}^{\infty} \frac{V_1^k}{k!}\nabla V_1 - \sum_{k=2}^{\infty} \frac{\nabla(V_1^k)}{k!} \|_{H^s} \lesssim \sum_{k=1}^{\infty} \frac{\|V_1^k\|_X}{k!} \|\nabla V_1\|_{H^s} + \sum_{k=2}^{\infty} \frac{\|\nabla(V_1^k)\|_{H^s}}{k!} \\
& \lesssim \sum_{k=1}^{\infty}\frac{c_1^k \|V_1\|_X^k}{ k!} \|\nabla V_1\|_{H^s} + \sum_{k=2}^{\infty} \frac{c_1^k \|V_1\|_X^k}{k!} \\
& \lesssim (\exp(c_1 \|V_1\|_X) - 1) \|\nabla V_1\|_{H^s}+( \exp (c_1 \|V_1\|_X) - 1 - c_1 \|V_1\|_X) \\
& \le \frac{c_0}{2} \delta^2 \, .
\end{align*}
These estimates imply bv (\ref{53'}) in the case $\delta = \epsilon  \ll 1$ :
\begin{equation}
\label{54'}
\|(T_1 A)^{cf}(0)\|_{H^s} \lesssim c_0 \epsilon \delta = c_0 \epsilon^2 \le \frac{1}{2} \epsilon^{\frac{3}{2}} \, .
\end{equation}
Moreover by (\ref{40})
\begin{equation} 
\label{**}
\|(T_1 A)(0)\|_{H^s} \le \|A^{df}(0)\|_{H^s} + c_0 \epsilon \delta \le \epsilon + \frac{1}{2} \epsilon^{\frac{3}{2}} \le 2 \epsilon \, ,
\end{equation}
and combining this with (\ref{54'}) :
\begin{equation} 
\label{56'}
\|(T_1 A)^{df}(0)\|_{H^s} \le \epsilon +  \epsilon^{\frac{3}{2}} \le 2 \epsilon \, .
\end{equation}
Similarly we also obtain by Lemma \ref{Lemma}
\begin{align*}
\|\partial_t(T_1 A)^{cf}(0)\|_{H^{s-1}} &\lesssim c_0 \epsilon \delta = c_0 \epsilon^2 \le \frac{1}{2} \epsilon^{\frac{3}{2}} \\
\|\partial_t (T_1 A)(0)\|_{H^{s-1}} &
\le \epsilon + \frac{1}{2}\epsilon^{\frac{3}{2}}  \\
\|\partial_t(T_1 A)^{df}(0)\|_{H^{s-1}} &\le \epsilon +  \epsilon^{\frac{3}{2}} \le 2 \epsilon \, ,
\end{align*}
and 
 $$\|\partial_t(T_1 \phi)(0)\|_{H^s} + \|(\partial_t T_1 \phi)(0)\|_{H^{s-1}} \le \epsilon + \frac{1}{2} \epsilon^{\frac{3}{2}} \, .$$
We have now shown that (\ref{***}) with $\epsilon$ replaced by $\epsilon+\frac{1}{2}\epsilon^{\frac{3}{2}}$ and (\ref{42}) with $\delta = \frac{1}{2}\epsilon^{\frac{3}{2}}$ are fulfilled with $A$ and $\phi$ replaced by $T_1 A$ and $T_1 \phi$.

In a next step we define $V_2 := -(-\Delta)^{-1} div (T_1 A)(0)$ so that $\nabla V_2 = (T_1A)^{cf}(0)$ and  thus by (\ref{54'}) 
\begin{equation}
\label{41}
\|V_2\|_X = \|\nabla V_2\|_{H^s} \le \epsilon^{\frac{3}{2}} \,. 
\end{equation}
We define the next gauge transform $T_2$ by
\begin{align*}
A & \longmapsto U_2 T_1A U_2^{-1} - \nabla U_2 U_2^{-1} \\
\phi & \longmapsto U_2 T_1 \phi U_2^{-1} 
\end{align*}
with $U_2 = \exp V_2$. \\
Calculating as above we obtain
\begin{align*}
(T_2A)(0) & = (T_1 A)^{df}(0) + (\exp V_2 (T_1A)^{df}(0) \exp(-V_2) - (T_1A)^{df}(0)) \\
& \hspace{1em} +( (\exp V_2 \nabla V_2  - \nabla(\exp V_2)) \exp(-V_2)) 
\end{align*}
where we used  $\nabla V_2 = (T_1A)^{cf}(0)$. This implies :
\begin{align*}
&\|(T_2 A)^{cf}(0)\|_{H^s} \\& \le c_2 (\|\exp V_2 (T_1 A)^{df}(0) (\exp(-V_2) - I)\|_{H^s} + \|(\exp V_2 - I) (T_1 A)^{df}(0)\|_{H^s} \\
& \hspace{1em} + \| ( (\exp V_2 \nabla V_2  - \nabla(\exp V_2)) \exp(-V_2))\|_{H^s})\, .
\end{align*}
The first two terms on the right hand side are bounded by  (\ref{41}) by
$$c_2((\exp(c_1 \|V_2\|_X )-1) + 1) 2\epsilon (\exp(c_1\|V_2\|_X) -1) \lesssim (\epsilon^{\frac{3}{2}} + 1) \epsilon \epsilon^{\frac{3}{2}} \lesssim \epsilon^{\frac{5}{2}} \le \frac{1}{4} \epsilon^2 \, , $$
where we used (\ref{56'}), whereas the last term on the right hand side can be handled similarly as in the first iteration step :
$$ c_2\| ( (\exp V_2 \nabla V_2  - \nabla(\exp V_2)) \exp(-V_2))\|_{H^s} \lesssim \epsilon^3 \le \frac{1}{4} \epsilon^2 \, . $$
This implies
$$ \|(T_2A)^{cf}(0)\|_{H^s} \le \frac{1}{2} \epsilon^2$$
and also
$$ \|(T_2A)(0)\|_{H^s} \le \|(T_1A)^{df}(0)\|_{H^s} + \frac{1}{2} \epsilon^2 \le \epsilon + \epsilon^{\frac{3}{2}}+ \frac{1}{2}\epsilon^2 \le 2\epsilon \, , $$
thus
$$\|(T_2A)^{df}(0)\|_{H^s} \le \epsilon + \epsilon^{\frac{3}{2}} + \epsilon^2 \le 2\epsilon \, . $$
Similar estimates are also obtained for $\|\partial_t(T_2A)^{cf}(0)\|_{H^{s-1}}$ , $  \|\partial_t(T_2A)(0)\|_{H^{s-1}}$ and $\|\partial_t (T_2A)^{df}(0)\|_{H^s}$ 
We also obtain
$$ \|(T_2 \phi)(0)\|_{H^s} + \|(\partial_t T_2 \phi)(0)\|_{H^{s-1}} \le \epsilon + \epsilon^{\frac{3}{2}} + \epsilon^2 \, . $$ 
We have now shown that (\ref{***}) with $\epsilon$ replaced by $\epsilon+\epsilon^{\frac{3}{2}}+\epsilon^2$ and (\ref{42}) with $\delta = \frac{1}{2}\epsilon^2$ are fulfilled with $A$ and $\phi$ replaced by $T_2 A$ and $T_2 \phi$ .

By iteration we obtain a sequence of gauge transforms $T_k$ defined by
\begin{align*}
A &\longmapsto \prod_{l=k}^1 \exp V_l A  \prod_{l=1}^k \exp(- V_l) - \nabla (\prod_{l=k}^1\exp V_l) \prod_{l=1}^k \exp(- V_l) \\
\phi &\longmapsto  \prod_{l=k}^1 \exp V_l \, \phi   \prod_{l=1}^k \exp(- V_l)
\end{align*}
with
$$ V_l := -(-\Delta)^{-1} div \,(T_{l-1}A)(0) \, $$
where $T_0 := id$. 
We remark that
$\nabla V_{k}= (T_{k-1} A)^{cf}(0)$. 
We now make the assumption that for some $k\ge 2$ we know that
$$\|(T_{k-1}A)^{df} (0)\|_{H^s} \le \epsilon + \epsilon^{\frac{3}{2}} + ... + \epsilon^{\frac{k+1}{2}} \le 2 \epsilon$$
and 
\begin{equation}
\label{58'}
\|V_k\|_X = \|(T_{k-1}A)^{cf} (0)\|_{H^s} \le \frac{1}{2} \epsilon^{\frac{k+1}{2}} \, . 
\end{equation}
This holds for the case $k=2$ as shown before.
Exactly as in the first two steps we obtain the estimate (with implicit constants independent of $k$ from now on) :
\begin{align*}
 &\|V_{k+1}\|_X = \|\nabla V_{k+1}\|_{H^s} = \|(T_k A)^{cf}(0)\|_{H^s} \\
&\lesssim ((\exp(c_1 \|V_k\|_X)-1) + 1) \|(T_{k-1}A)^{df}(0)\|_{H^s} (\exp(c_1 \|V_k\|_X -1) \\
&\hspace{1em} + (\exp(c_1\|V_k\|_X) -1)\|\nabla V_k\|_{H^s} + (\exp(c_1 \|V_k\|_X) -1 - c_1 \|V_k\|_X) \\
&\lesssim (\|(T_{k-1}A)^{cf}(0)\|_{H^s} + 1) \|(T_{k-1}A)^{df}(0)\|_{H^s} \|(T_{k-1}A)^{cf}(0)\|_{H^s} \\ 
&\hspace{1em}+  \|(T_{k-1}A)^{cf}(0)\|_{H^s} \|(T_{k-1}A)^{df}(0)\|_{H^s} + \|(T_{k-1}A)^{cf}(0)\|_{H^s}^2 \\
&\lesssim (\epsilon^{\frac{k+1}{2}}+1) 2\epsilon \epsilon^{\frac{k+1}{2}} + \epsilon^{\frac{k+1}{2}}\epsilon^{\frac{k+1}{2}}+\epsilon^{k+1} 
\lesssim \epsilon^{\frac{k+3}{2}} + \epsilon^{k+1} \le \frac{1}{2} \epsilon^{\frac{k}{2}+1}
\end{align*}
and
$$
\|(T_kA)(0)\|_{H^s} 
\le \|(T_{k-1}A)^{df}(0)\|_{H^s} + \frac{1}{2} \epsilon^{\frac{k}{2} +1} \le \epsilon + \epsilon^{\frac{3}{2}} + ... + \epsilon^{\frac{k+1}{2}} + \frac{1}{2} \epsilon^{\frac{k}{2}+1} \le 2\epsilon \, ,$$
thus
\begin{equation}
\label{CC}
 \|(T_kA)^{df}(0)\|_{H^s} \le \epsilon+\epsilon^{\frac{3}{2}} + ... + \epsilon^{\frac{k}{2}+1} \le 2 \epsilon \, . 
\end{equation}
Thus these estimates hold for any $k \ge 2$.
Similarly one can show that
\begin{align}
\nonumber
&\|(\partial_t T_kA)(0)\|_{H^{s-1}}+ \|(\partial_t T_kA)^{df}(0)\|_{H^{s-1}} +\|(T_k \phi)(0)\|_{H^s} +
\|(\partial_t T_k \phi)(0)\|_{H^{s-1}} \\ 
\label{CCC}
& \hspace{1em} \lesssim \epsilon \, .
\end{align}
Next we estimate
\begin{align*}
\|T_k A\|_{H^s} &\le \|(\prod_{l=k}^1 \exp V_l) A \prod_{l=1}^k \exp(-V_l)\|_{H^s} + \|\nabla(\prod_{l=k}^1 \exp V_l) \prod_{l=1}^k \exp(-V_l)  \|_{H^s} \\
& = I + II \, . 
\end{align*}
We further estimate
\begin{align*}
I &\le \|A\|_{H^s} + \|((\prod_{l=k}^1 \exp V_l)- I)  A ((\prod_{l=1}^k \exp(-V_l)-I)\|_{H^s} \\ &\hspace{1em} + \|  A ((\prod_{l=1}^k \exp(-V_l)-I)\|_{H^s} 
+ \|((\prod_{l=k}^1 \exp V_l)- I)  A \|_{H^s} \\
&= \|A\|_{H^s}+ I_1 + I_2 + I_3
\end{align*}
In order to control $I_1$ we consider first
\begin{align}
\nonumber
&\| \prod_{l=k}^1 \exp V_l - I \|_X = \|\prod_{l=k}^1 \sum_{n=0}^{\infty} \frac{V_l^n}{n!} - I \|_X = \| \sum_{m=1}^{\infty} \sum _{n_1+...+n_k = m} \prod_{l=k}^1 \frac{V_l^{n_l}}{n_l !} \|_X \\
\label{C}
&\lesssim  \sum_{m=1}^{\infty} \sum _{n_1+...+n_k = m} \prod_{l=k}^1 \frac{(c_1 \|V_l\|_X)^{n_l}}{n_l !}   = \prod_{l=k}^1 \exp(c_1 \|V_l\|_X) - 1 \\
\nonumber
&=\exp(\sum_{l=1}^k c_1 \|V_l\|_X) - 1 \lesssim \exp(\sum_{l=1}^k c_1 \epsilon^{\frac{l+1}{2}}) - 1 \lesssim \exp(2c_1 \epsilon) - 1 \lesssim \epsilon 
\end{align}
independently of $k$  where we used (\ref{58'}). Consequently
$$
I_1 \lesssim \|(\prod_{l=k}^1 \exp V_l)- I\|_X \| A\|_{H^s} \| (\prod_{l=1}^k \exp(-V_l)-I\|_X \lesssim \|A\|_{H^s} \epsilon^2\, . $$
Estimating $I_2$ and $I_3$ similarly we obtain
$$ I \lesssim \|A\|_{H^s} (1+ \epsilon^2 + \epsilon) \, .$$
Moreover
\begin{align*}
II & \le \| \nabla(\prod_{l=k}^1 \exp V_l - I)(\prod_{l=1}^k \exp(-V_l) - I ) \|_{H^s} + \|\nabla
(\prod_{l=k}^1 \exp V_l - I)\|_{H^s} \\
& \lesssim \|\prod_{l=k}^1 \exp V_l - I\|_X \|(\prod_{l=1}^k \exp(-V_l) - I ) \|_X + \|
(\prod_{l=k}^1 \exp V_l - I)\|_X \\
& \lesssim \epsilon^2 + \epsilon \, ,
\end{align*}
Summarizing we obtain with implicit constants which are independent of $k$ :
$$
\|T_k A\|_{H^s} \
 \lesssim \|A\|_{H^s} + \epsilon $$ 
Similarly we also obtain
$$
\|\partial_t( T_k A )\|_{H^{s-1}}  \lesssim \|\partial_t A\|_{H^{s-1}} $$
and
$$ \|T_k \phi\|_{H^s} \lesssim \|\phi\|_{H^s} \quad , \quad \|\partial_t (T_k \phi)\|_{H^{s-1}} \lesssim \|\partial_t \phi\|_{H^{s-1}} \, . $$
We want to consider the mapping $T$ defined by $TA = \lim_{k\to \infty} T_k A$ and $T\phi = \lim_{k\to \infty} T_k \phi$ , where the limit is taken in $C^0([0,1],H^s) \cap C^1([0,1],H^{s-1})$. This would imply by (\ref{58'}): $\|(TA)^{cf}(0)\|_{H^s} = \lim_{k \to \infty} \|(T_kA)^{cf}\|_{H^s} = 0$ , thus the desired property $$(TA)^{cf}(0)=0 \, .$$
Now define
$$ SA := \prod_{l=\infty}^1(\exp V_l)  A \prod_{l=1}^{\infty} \exp(-V_l)- \nabla (\prod_{l=\infty}^1 \exp V_l)  \prod_{1=1}^{\infty} \exp(-V_l) = UAU^{-1} - \nabla U U^{-1}\, ,$$
with $U:=\prod_{l=\infty}^1 \exp V_l$, where the limit is taken with respect to $\| \cdot \|_X$ . \\
This limit in fact exists, because by the calculations in (\ref{C}) we obtain for $N > k$ the estimate
$$ \| \prod_{l=N}^1 \exp V_l - \prod_{l=k}^1 \exp V_l \|_X \lesssim \|\prod_{l=N}^{k+1} \exp V_l - I \|_X (\|\prod_{l=k}^1 \exp V_l - I\|_X + 1)\lesssim \epsilon^{\frac{k}{2}+1} (\epsilon +1) \, . $$
We also obtain $U^{-1}=\prod_{l=1}^{\infty} \exp(-V_l)$, which is defined in the same way. \\
In order to prove $S=T$ we estimate as follows :
\begin{align*}
&\|SA-T_kA\|_{H^s} \\
& \le \|(\prod_{l=\infty}^1\exp V_l-\prod_{l=k}^1 \exp V_l) A \prod_{l=1}^{\infty}\exp(-V_l)\|_{H^s} \\
& + \|\prod_{l=k}^1\exp V_l A(\prod_{l=1}^{\infty}\exp(- V_l)-\prod_{l=1}^k \exp(- V_l))\|_{H^s} \\
& + \|\nabla(\prod_{l=\infty}^1(\exp V_l)-\prod_{l=k}^1 \exp V_l) (\prod_{l=1}^{\infty} \exp(- V_l)\|_{H^s} \\
& + \|\nabla(\prod_{l=k}^1\exp V_l(\prod_{l=1}^{\infty} \exp(- V_l) - \prod_{l=1}^k \exp(- V_l))\|_{H^s} \\
& = I + II + III + IV
\end{align*}
Now
\begin{align*}
I & =\| (\prod_{l=\infty}^{k+1} \exp V_l - I)  \prod_{l=k}^1 \exp V_l A \prod_{l=1}^{\infty} \exp(-V_l) \|_{H^s} \\
& \lesssim \|\prod_{l=\infty}^{k+1} \exp V_l - I \|_X  (\| \prod_{l=k}^1 \exp V_l -I\|_X + 1) \| A \prod_{l=1}^{\infty} \exp(-V_l)  \|_{H^s}\, .
\end{align*}
Now by (\ref{C}) we obtain
$$\| \prod_{l=k}^1 \exp V_l -I\|_X \lesssim \epsilon \, , $$
and
$$ \|A \prod_{l=1}^{\infty} \exp (-V_l) \|_{H^s} \lesssim \|A\|_{H^s} (1+\| \prod_{l=1}^{\infty} \exp (-V_l) -I\|_X)  \lesssim \|A\|_{H^s} (1+\epsilon) $$
and also similarly as in (\ref{C})
$$ \| \prod_{l=\infty}^{k+1} \exp(- V_l) -I\|_X \le \exp(\sum_{l=k+1}^{\infty} c_1 \epsilon^{\frac{k}{2}+1}) - 1 \lesssim \exp(c \epsilon^{\frac{k}{2}+1}) -1\lesssim \epsilon^{\frac{k}{2} + 1} \, $$
so that
$$ I + II \lesssim \epsilon^{\frac{k}{2}+1} \|A\|_{H^s} \, . $$
Next we estimate
\begin{align*}
III & \le  \|\nabla(\prod_{l=\infty}^1(\exp V_l)-\prod_{l=k}^1 \exp V_l) ((\prod_{l=1}^{\infty} \exp(- V_l)- I)\|_{H^s} \\
&\hspace{1em}+\|\nabla(\prod_{l=\infty}^1(\exp V_l)-\prod_{l=k}^1 \exp V_l)\|_{H^s} \\
& \lesssim \|I - \prod_{l=\infty}^{k+1} \exp V_l \|_X (\|\prod_{l=k}^1 \exp V_l - I\|_X + 1) (\|\prod_{l=1}^{\infty} \exp(-V_l) - I\|_X +1) \\
& \lesssim \epsilon^{\frac{k}{2}+1} (\epsilon +1)(\epsilon +1) \lesssim \epsilon^{\frac{k}{2}+1} \, .
\end{align*}	
Finally
\begin{align*}
IV &\le \|\nabla(\prod_{l=k}^1 \exp V_l- I) \prod_{l=1}^k \exp(- V_l) (I - \prod_{l=k+1}^{\infty} \exp(- V_l))\|_{H^s} \\
& \lesssim \|\prod_{l=k}^1 \exp V_l- I\|_X (\prod_{l=1}^k \exp(- V_l) -I\|_X + 1) \|I - \prod_{l=k+1}^{\infty} \exp(- V_l)\|_X \\
& \lesssim \epsilon  (\epsilon +1) \epsilon^{\frac{k}{2}+1} \lesssim \epsilon^{\frac{k}{2}+2} \, ,
\end{align*}
so that we obtain
$$ \|SA - T_kA\|_{H^s} \lesssim \epsilon^{\frac{k}{2}+1} (\|A\|_{H^s} + 1)\, \rightarrow \, 0
\quad (k \to \infty) \, , $$
thus $T_kA \to SA $ in $C^0([0,1],H^s)$ and similarly $\partial_t T_k A \to \partial_t SA$ in $C^0([0,1],H^{s-1})$ as well as $T_k \phi \to S\phi$ in $C^0([0,1],H^s)$ and $\partial_t T_k \phi \to \partial_t S\phi$ in $C^0([0,1],H^{s-1}) \,.$ We have shown that $T=S$ is a gauge transformation which besides fulfilling the temporal gauge has the property $(TA)^{cf}(0)=0$ and preserves the regularity $A,\phi \in C^0([0,1],H^s)\cap C^1([0,1],H^{s-1})$. 
From the properties (\ref{CC}) and (\ref{CCC}) of $T_k$ we also deduce
$$ \|(TA)^{df}(0)\|_{H^s} + \|(\partial_t TA)^{df}(0)\|_{H^{s-1}} + \|(T\phi)(0)\|_{H^s} + \|(\partial_t T\phi)(0)\|_{H^{s-1}} \lesssim \epsilon \, . $$
Assume now that
$ A = A_- + A_+ + A' $ , where $A_{\pm} \in X^{s,\frac{3}{4}+}_{\pm}[0,1]$ , $A' \in X^{s+\alpha,\frac{1}{2}+}_{\tau=0}[0,1]$ and $\partial_t A' \in C^0([0,1],H^{s-1})$ . Let $$TA=U A U^{-1} -\nabla U U^{-1} \, ,$$ where $U = \prod_{l=\infty}^1 \exp V_l$ , is defined as above. We want to show that $TA$ has the same regularity. Let $\psi=\psi(t)$ be a smooth function with $\psi(t)=1$ for $0 \le t \le 1$ and $\psi(t)=0$ for $t\ge 2$. Then we obtain by Lemma \ref{Lemma1} below and (\ref{C}) :
\begin{align*}
\|U A_{\pm} \psi\|_{X^{s,\frac{3}{4}+}_{\pm}} & \lesssim \|\nabla U \psi\|_{X^{s,1}_{\pm}} \|A_{\pm}\|_{X^{s,\frac{3}{4}+}_{\pm}} \lesssim \|\nabla U\|_{H^s} \|A_{\pm}\|_{X^{s,\frac{3}{4}+}_{\pm}} \\
&\lesssim \| U -I\|_X \|A_{\pm}\|_{X^{s,\frac{3}{4}+}_{\pm}} \lesssim \epsilon \|A_{\pm}\|_{X^{s,\frac{3}{4}+}_{\pm}} \, ,
\end{align*}
thus
$$ \|U A_{\pm}\|_{X^{s,\frac{3}{4}+}_{\pm}[0,1]} \lesssim \epsilon \|A_{\pm}\|_{X^{s,\frac{3}{4}+}_{\pm}[0,1]} \, . $$
Similarly we obtain
$$ \|U A_{\pm} U^{-1}\|_{X^{s,\frac{3}{4}+}_{\pm}[0,1]} \lesssim \epsilon \|U A_{\pm}\|_{X^{s,\frac{3}{4}+}_{\pm}}      \lesssim \epsilon^2\|A_{\pm}\|_{X^{s,\frac{3}{4}+}_{\pm}[0,1]} < \infty \, . $$
We also have
\begin{align*}
&\|(\nabla U)\psi U^{-1}\psi \|_{X^{s,\frac{3}{4}+}_{\pm}} \lesssim \|\nabla U \psi\|_{X^{s,\frac{3}{4}+}_{\pm}} \|\nabla ( U^{-1}) \psi\|_{X^{s,1}_{\pm}} \lesssim \|\nabla U\|_{H^s} \|\nabla(U^{-1})\|_{H^s} \, ,                                          
\end{align*}
thus
\begin{align*}
&\|(\nabla U) U^{-1}\|_{X^{s,\frac{3}{4}+}_{\pm}[0,1]} \lesssim \|\nabla U\|_{H^s} \|\nabla(U^{-1})\|_{H^s} \, .                                          
\end{align*}
Moreover by Sobolev we obtain
\begin{align*}
\|U A' \psi\|_{X^{s+\alpha,\frac{1}{2}+}_{\tau=0}} & \lesssim \| \nabla(U) \psi\|_{X^{s,1}_{\tau=0}} \|A'\|_{X^{s+\alpha,\frac{1}{2}+}_{\tau=0}} \\
& \lesssim \| \nabla U \|_{H^s} \|A'\|_{X^{s+\alpha,\frac{1}{2}+}_{\tau=0}} \lesssim \epsilon  \|A'\|_{X^{s+\alpha,\frac{1}{2}+}_{\tau=0}} \, . 
\end{align*}
Similarly as before this implies
$$ \|U  A' U^{-1}\|_{X^{s+\alpha,\frac{1}{2}+}_{\tau=0}[0,1]} \lesssim \epsilon^2\|A'\|_{X^{s+\alpha,\frac{1}{2}+}_{\tau=0}[0,1]} < \infty \, . $$
By Sobolev's muliplication law we also obtain
$$ \|U \partial_t A'\|_{C^0([0,1],H^{s-1})} \lesssim \|\nabla U\|_{H^s} \|\partial_t A'\|_{C^0([0,1],H^{s-1})} \lesssim \epsilon \|\partial_t A'\|_{C^0([0,1],H^{s-1})} \, . $$
As before this implies
$$\|U \partial_tA' U^{-1}\|_{C^0([0,1],H^{s-1})} \lesssim \epsilon^2 \|\partial_t A'\|_{C^0([0,1],H^{s-1})} < \infty \, .$$
We have thus shown that $TA$ has the same regularity as $A$.
The same estimates also show that
$$ \|U \phi_{\pm} U^{-1}\|_{X^{s,\frac{3}{4}+}_{\pm}[0,1]}      \lesssim \epsilon^2\|\phi_{\pm}\|_{X^{s,\frac{3}{4}+}_{\pm}[0,1]} < \infty \, , $$
so that $T\phi = U \phi U^{-1}$ maps $X^{s+\frac{1}{4},\frac{1}{2}+}_+ + X^{s+\frac{1}{4},\frac{1}{2}+}_-$ into itself. The same properties also hold for its inverse $T^{-1}$ which is given by
\begin{align*}
B &\longmapsto U^{-1} B U + U^{-1} \nabla U \\
\phi' &\longmapsto U \phi' U^{-1} \, .
\end{align*}
\end{proof}

In the last proof we used the following
\begin{lemma}
\label{Lemma1}
The following estimate holds for $s>\frac{n}{2}-\frac{3}{4}$ and $\epsilon >0$ sufficiently small:
$$ \|uv\|_{X^{s,\frac{3}{4}+\epsilon}_{\pm}} \lesssim \|\nabla u\|_{X^{s,1}_{\pm}} \|v\|_{X^{s,\frac{3}{4}+\epsilon}_{\pm}} \, . $$
\end{lemma}
\begin{proof}
By Tao \cite{T}, Cor. 8.2 we may replace $\nabla$ by $\langle \nabla \rangle$ so that it suffices to prove
$$  \|uv\|_{X^{s,\frac{3}{4}+\epsilon}_{\pm}} \lesssim \| u\|_{X^{s+1,1}_{\pm}} \|v\|_{X^{s,\frac{3}{4}+\epsilon}_{\pm}} \, . $$
We start with the elementary estimate
$$|(\tau_1 + \tau_2)\mp |\xi_1+\xi_2|| \le |\tau_1 \mp |\xi_1|| + |\tau_2 \mp |\xi_2|| + |\xi_1| + |\xi_2| - |\xi_1+\xi_2| \, . $$
Assume now w.l.o.g. $|\xi_2|\ge|\xi_1|$. We have 
$$|\xi_1|+|\xi_2|-|\xi_1+\xi_2| \le |\xi_1|+|\xi_2| + |\xi_1| - |\xi_2| =  2|\xi_1| \, ,$$
so that $$|(\tau_1 + \tau_2)\mp |\xi_1+\xi_2||  \le |\tau_1 \mp\xi_1| + |\tau_2 \mp |\xi_2|| + 2\min(|\xi_1|,|\xi_2|) \, . $$
Using Fourier transforms by standard arguments it thus suffices to show the following three estimates:
\begin{align*}
\|uv\|_{X_{\pm}^{s,0}} & \lesssim \|u\|_{X^{s+1,\frac{1}{4}-\epsilon}_{\pm}} \|v\|_{X^{s,\frac{3}{4}+\epsilon}_{\pm}} \\
\|uv\|_{X_{\pm}^{s,0}} & \lesssim \|u\|_{X^{s+1,1}_{\pm}} \|v\|_{X^{s,0}_{\pm}} \\
\|uv\|_{X_{\pm}^{s,0}} & \lesssim \|u\|_{X^{s+\frac{1}{4}-\epsilon,1}_{\pm}} \|v\|_{X^{s,\frac{3}{4}+\epsilon}_{\pm}} 
\end{align*}
The first and second estimate easily follow from Sobolev, whereas the last one is implied by \cite{FK} , Thm. 1.1.  
\end{proof}

\section{Proof of Theorem \ref{Theorem1} and Theorem \ref{Theorem1'}}
\begin{proof}
We only prove Theorem \ref{Theorem1}.
It suffices to construct a unique local solution of (\ref{3}),(\ref{4}),(\ref{5}) with initial conditions
$$ A^{df}(0) = a^{df}  \, , \, (\partial_t A^{df})(0) = {a'}^{df}  \, , \,
A^{cf}(0) = a^{cf} \, , \,\phi(0)=\phi_0 \, , \, (\partial_t \phi)(0) = \phi_1 \, ,$$
which fulfill
$$ \|A^{df}(0)\|_{H^s} + \|(\partial_t A)^{df}(0)\|_{H^{s-1}} + \|A^{cf}(0)\|_{H^s} + \|\phi(0)\|_{H^s} + \|(\partial_t \phi)(0)\|_{H^{s-1}} \le \epsilon $$
for a sufficiently small $\epsilon > 0$.
By Lemma \ref{Lemma} there exists a gauge transformation $T$ which fulfills (\ref{T1}) and $(TA)^{cf}(0) =0$. We use Prop. \ref{Prop} to construct a unique solution $(\tilde{A},\tilde{\phi})$ of (\ref{3}),(\ref{4}),(\ref{5}) , where $\tilde{A}=\tilde{A}_+^{df} + \tilde{A}_-^{df} +\tilde{A}^{cf}$ and $\tilde{\phi} = \tilde{\phi}_+ + \tilde{\phi}_-$ ,  with data
$$\tilde{A}^{df}(0)= (TA)^{df}(0) \, , \, (\partial_t \tilde{A})^{df}(0) = (\partial_t (TA)^{df})(0) \, , \, \tilde{A}^{cf}(0)= (TA)^{cf}(0)=0 \, ,$$ 
$$ \, \tilde{\phi}(0) = (T\phi)(0) \, , \, (\partial_t \tilde{\phi})(0) = (\partial_t T \phi)(0) $$
with the regularity
$$ \tilde{A}^{df}_{\pm} \in X^{s,\frac{3}{4}+}_{\pm}[0,1]  ,  \tilde{A}^{cf} \in X^{s+\alpha,\frac{1}{2}+}_{\tau=0}[0,1]  ,  \partial_t \tilde{A}^{cf} \in C^0([0,1],H^{s-1})  , 
 \tilde{\phi}_{\pm} \in  X^{s,\frac{3}{4}+}_{\pm}[0,1] \, . $$
This solution satisfies also $\tilde{A},\tilde{\phi} \in C^0([0,1],H^s) \cap C^1[0,1],H^{s-1})$.

Applying the inverse gauge transformation $T^{-1}$ according to Lemma \ref{Lemma} we obtain a unique solution of (\ref{3}),(\ref{4}),(\ref{5}) with the required initial data and also the same regularity. 

 The proof of Theorem \ref{Theorem1'} is completely analogous by use of Prop. \ref{Prop'}.
\end{proof}

\end{document}